\documentclass[a4paper,twoside,11pt,english,intlimits]{article}

\usepackage{amsmath,amsthm,amsfonts,amssymb}
\usepackage{babel}
\usepackage[T1]{fontenc}
\usepackage{times,euler,euscript,stmaryrd}
\usepackage{fancyhdr,titlesec,url,enumerate}
\usepackage[all]{xy}
\usepackage{hyperref}
\usepackage{ifthen}
\usepackage{tikz}
\usepackage{comment}
\usepackage{wasysym}

\CompileMatrices
\usetikzlibrary{calc}

\titleformat{\section}[block]{\scshape\filcenter\Large}{\thesection.}{.5em}{}
\titleformat{\subsection}[block]{\bfseries\filcenter\large}{\thesubsection.}{.5em}{\medskip}
\titleformat{\subsubsection}[runin]{\bfseries}{\thesubsubsection.}{.5em}{}[.]

\titlespacing{\subsubsection}{0pt}{10pt}{.5em}

\newtheoremstyle{ntheorem}%
	{\topsep}{\topsep}{\itshape}{0pt}{\bfseries}{.}{.5em}%
	{\thmnumber{#2.\hspace{.5em}}\thmname{#1}\thmnote{ (#3)}}
	
\newtheoremstyle{ndefinition}%
	{\topsep}{\topsep}{\normalfont}{0pt}{\bfseries}{.}{.5em}%
	{\thmnumber{#2.\hspace{.5em}}\thmname{#1}\thmnote{ (#3)}}
	
\newtheoremstyle{nremark}%
	{\topsep}{\topsep}{\normalfont}{0pt}{\itshape}{.}{.5em}%
	{\thmnumber{}\thmname{#1}\thmnote{ (#3)}}

\theoremstyle{ntheorem}
  	\newtheorem{theorem}[subsubsection]{Theorem}
  	\newtheorem{proposition}[subsubsection]{Proposition}
	\newtheorem{lemma}[subsubsection]{Lemma}

\theoremstyle{ndefinition}

	\newtheorem{example}[subsubsection]{Example}
         \newtheorem{remark}[subsubsection]{Remark}

\fancypagestyle{plain}{
\fancyhf{}\fancyfoot[LC,RC]{}
\fancyfoot[LE,RO]{$\thepage$}

}

\UseTips
\SelectTips{eu}{11}

\newdir{ >}{{}*!/-10pt/@{>}}
\newdir{ -}{{}*!/-10pt/@{}}
\newdir{> }{{}*!/+10pt/@{>}}

\makeatletter

\xyletcsnamecsname@{dir4{}}{dir{}}
\xydefcsname@{dir4{-}}{\line@ \quadruple@\xydashh@}
\xydefcsname@{dir4{.}}{\point@ \quadruple@\xydashh@}
\xydefcsname@{dir4{~}}{\squiggle@ \quadruple@\xybsqlh@}
\xydefcsname@{dir4{>}}{\Tttip@}
\xydefcsname@{dir4{<}}{\reverseDirection@\Tttip@}

\xydef@\quadruple@#1{%
	\edef\Drop@@{%
		\dimen@=#1\relax
		\dimen@=.5\dimen@
		\A@=-\sinDirection\dimen@
		\B@=\cosDirection\dimen@
		\setboxz@h{%
			\setbox2=\hbox{\kern3\A@\raise3\B@\copy\z@}%
			\dp2=\z@ \ht2=\z@ \wd2=\z@ \box2
			\setbox2=\hbox{\kern\A@\raise\B@\copy\z@}%
			\dp2=\z@ \ht2=\z@ \wd2=\z@ \box2
			\setbox2=\hbox{\kern-\A@\raise-\B@\copy\z@}%
			\dp2=\z@ \ht2=\z@ \wd2=\z@ \box2
			\setbox2=\hbox{\kern-3\A@\raise-3\B@ \noexpand\boxz@}%
			\dp2=\z@ \ht2=\z@ \wd2=\z@ \box2
		}%
		\ht\z@=\z@ \dp\z@=\z@ \wd\z@=\z@ \noexpand\styledboxz@
	}%
}

\xydef@\Tttip@{\kern2pt \vrule height2pt depth2pt width\z@
	\Tttip@@ \kern2pt \egroup
	\U@c=0pt \D@c=0pt \L@c=0pt \R@c=0pt \Edge@c={\circleEdge}%
	\def\Leftness@{.5}\def\Upness@{.5}%
	\def\Drop@@{\styledboxz@}\def\Connect@@{\straight@{\dottedSpread@\jot}}}
	
\xydef@\Tttip@@{%
	\dimen@=.25\dimen@
 	\B@=\cosDirection\dimen@
	\setboxz@h\bgroup\reverseDirection@\line@ \wdz@=\z@ \ht\z@=\z@ \dp\z@=\z@
	{\vDirection@(1,-1)\xydashl@ \xyatipfont\char\DirectionChar}%
	{\vDirection@(1,+1)\xydashl@ \xybtipfont\char\DirectionChar}%
}

\xydef@\ar@form{
	\ifx \space@\next \expandafter\DN@\space{\xyFN@\ar@form}%
	\else\ifx ^\next \DN@ ^{\xyFN@\ar@style}\edef\arvariant@@{\string^}%
	\else\ifx _\next \DN@ _{\xyFN@\ar@style}\edef\arvariant@@{\string_}%
	\else\ifx 0\next \DN@ 0{\xyFN@\ar@style}\def\arvariant@@{0}%
	\else\ifx 1\next \DN@ 1{\xyFN@\ar@style}\def\arvariant@@{1}%
	\else\ifx 2\next \DN@ 2{\xyFN@\ar@style}\def\arvariant@@{2}%
	\else\ifx 3\next \DN@ 3{\xyFN@\ar@style}\def\arvariant@@{3}%
	\else\ifx 4\next \DN@ 4{\xyFN@\ar@style}\def\arvariant@@{4}%
	\else\ifx \bgroup\next \let\next@=\ar@style
	\else\ifx [\next \DN@[##1]{\ar@modifiers{[##1]}}
	\else\ifx *\next \DN@ *{\ar@modifiers}%
	\else\addLT@\ifx\next \let\next@=\ar@slide
	\else\ifx /\next \let\next@=\ar@curveslash
	\else\ifx (\next \let\next@=\ar@curveinout 
	\else\addRQ@\ifx\next \addRQ@\DN@{\ar@curve@}%
	\else\addLQ@\ifx\next \addLQ@\DN@{\xyFN@\ar@curve}%
	\else\addDASH@\ifx\next \addDASH@\DN@{\defarstem@-\xyFN@\ar@}%
	\else\addEQ@\ifx\next \addEQ@\DN@{\def\arvariant@@{2}\defarstem@-\xyFN@\ar@}%
	\else\addDOT@\ifx\next \addDOT@\DN@{\defarstem@.\xyFN@\ar@}%
	\else\ifx :\next \DN@:{\def\arvariant@@{2}\defarstem@.\xyFN@\ar@}%
	\else\ifx ~\next \DN@~{\defarstem@~\xyFN@\ar@}%
	\else\ifx !\next \DN@!{\dasharstem@\xyFN@\ar@}%
	\else\ifx ?\next \DN@?{\ar@upsidedown\xyFN@\ar@}%
	\else \let\next@=\ar@error
	\fi\fi\fi\fi\fi\fi\fi\fi\fi\fi\fi\fi\fi\fi\fi\fi\fi\fi\fi\fi\fi\fi\fi \next@}

\makeatother

\newcommand{\dfl}{\Rightarrow}

\newcommand{\qfl}{\xymatrix@1@C=10pt{\ar@4 [r] &}}

\newcommand{\opfl}[1]{\xymatrix @C=1.5em {\strut \ar@{->>} [r] ^-{#1} & \strut}}
\newcommand{\odfl}[1]{\xymatrix @C=1.5em {\strut \ar@2 [r] ^-{#1} & \strut}}
\newcommand{\otfl}[1]{\xymatrix @C=1.5em {\strut \ar@3 [r] ^-{#1} & \strut}}
\newcommand{\oqfl}[1]{\xymatrix @C=1.5em {\strut \ar@4 [r] ^*+{#1} & \strut}}

\newcommand{\ens}[1]{\left\{ #1 \right\}}

\newcommand{\ie}{\emph{i.e.}}
\renewcommand{\phi}{\varphi}
\renewcommand{\epsilon}{\varepsilon}

\definecolor{orange}{rgb}{1,0.55,0}
\definecolor{vert}{rgb}{0,0.45,0}

\newcommand{\ifthen}[2]{\ifthenelse{#1}{#2}{}}

\renewcommand{\to}{\longrightarrow}

\begin{document}

\thispagestyle{empty}

\begin{center}

\begin{Large}
{Finite convergent presentation of plactic monoid for type C\footnote{This work is partially supported by the French National Research Agency, ANR-13-BS02-0005-02.}}
\end{Large}
\bigskip\hrule height 1.5pt \bigskip

\begin{large}\begin{uppercase}
{Nohra Hage}
\end{uppercase}\end{large}

\vspace*{\stretch{1}}

\begin{small}\begin{minipage}{14cm}
\noindent\textbf{Abstract -- }
 We give an explicit presentation for the plactic monoid for type C using admissible column generators. Thanks to  the combinatorial properties of symplectic tableaux,  we prove that this presentation is finite and convergent. We obtain as a corollary that plactic monoids for type C satisfy homological finiteness properties.

\bigskip\noindent\textbf{Keywords --} Plactic monoid; crystal graphs; symplectic tableau;  convergent presentations.

\end{minipage}\end{small}

\vspace*{\stretch{1}}

\begin{small}\begin{minipage}{14cm}
\renewcommand{\contentsname}{}
\setcounter{tocdepth}{2}
\tableofcontents
\end{minipage}
\end{small}

\vspace*{\stretch{1}}
\end{center}

\today
\newpage

\section{Introduction}
\label{introduction}

The plactic monoid was discovered by Knuth in~\cite{Knuth70},  using the theory of  tableaux defined  by Schensted in~\cite{Schensted61} in his study of the longest increasing subsequence of a word.  Lascoux and Sch\"{u}tzenberger in~\cite{LascouxSchutzenberger81} used the plactic monoid to give a proof of the Littlewood--Richardson rule for the decomposition of  tensor products of irreducible modules on the Lie algebra of traceless square matrices. The plactic monoid has found several applications in algebraic combinatorics and representation theory~\cite{LascouxSchutzenberger81,LascouxLeclercThibon95,Fulton97,Lothaire02}. More recently, the plactic monoid was investigated by rewriting methods~\cite{KubatOkninski14, BokutChenChenLi15, CainGrayMalheiro15}.

Consider the ordered alphabet $\mathcal{A}_{n} = \{1<2<\ldots <n\}$. For every word $w$ over the free monoid $\mathcal{A}_{n}^{\ast}$, a unique tableau $P(w)$ can be computed using  Schensted's insertion algorithm (column insertion)~\cite{Schensted61}. One can define a relation $\sim$ on the free monoid $\mathcal{A}_{n}^{\ast}$ by:
\[
u\sim v \text{ }\text{ if and only if } \text{} P(u)=P(v)
\]
for all $u$ and $v$ in $\mathcal{A}_{n}^{\ast}$. Then the quotient $\mathbf{P_{n}}(A):=\mathcal{A}_{n}^{\ast}/\sim$ is called the  plactic  monoid. The plactic monoid can be  also described as the quotient of $\mathcal{A}_{n}^{\ast}$ by the congruence generated by the Knuth relations:
\begin{equation}
\label{Knuthrelations}
\big\{\;  xzy = zxy \;\big|\;1\leq x<y\leq z\leq n  \;\big\}
\;\cup \;
\big\{\;  yxz = yzx  \;\big|\;1\leq x\leq y < z \leq n 
\;\big\}
\end{equation}
which is called the \emph{Knuth presentation}. 
\bigskip

Thanks to Kashiwara's theory of crystal bases~\cite{Kashiwara91, JimboMisraMiwaOkado91, KashiwaraNakashima94, Kashiwara94},  plactic monoids can be defined for all classical simple Lie algebras.  To each classical simple Lie algebra, one associates a finite alphabet $S$ indexing a basis of the vector representation $V$ of the algebra. Two words $u$ and $v$ in the monoid $S^{\ast}$ are plactic congruent if they appear in the same place in isomorphic connected components of the crystal graph of the  representation $\underset{l}{\bigoplus}V^{\otimes l}$. 

The plactic monoid introduced by Schensted and Knuth corresponds to the representations of the simple Lie algebra of traceless square matrices which is of type A, and known as the plactic monoid of type A.  Similarly, plactic monoids of type C, B and D correspond  respectively to the representations of the symplectic Lie algebra,  the odd-dimensional  orthogonal Lie algebra  and the even-dimensional orthogonal Lie algebra.

Lascoux, Leclerc and Thibon defined in~\cite{LascouxLeclercThibon95} the plactic monoid of type A using the theory of crystal bases and gave a presentation of the plactic monoid $\mathbf{P_{n}}(C)$ of type C  without proof.  Lecouvey in~\cite{Lecouvey02} and Baker in~\cite{Baker00} described independently the monoid $\mathbf{P_{n}}(C)$ using also Kashiwara's theory of crystal bases.

Plactic monoids can be also defined for any semisimple Lie algebra  using Littelmann's path model, see~\cite{Littelmann96}. 
\bigskip

We deal with presentations of monoids from the rewriting theory perspective. In this context, relations are oriented and are considered as rewriting steps. A presentation terminates if it has no infinite rewriting sequence. A terminating presentation is confluent if all its critical branchings resolve. A presentation is convergent if it terminates and is confluent. Having a  finite convergent presentation  of a monoid has many advantages: for examples the computation of normal form and the computation of a free finitely generated resolution of the monoid which allows deduction of  some homological properties~\cite{Kobayashi90,Anick86}. An open problem was to find a finite convergent presentation of plactic monoids.

In~\cite{CainGrayMalheiro15},  Cain, Gray and  Malheiro answered positively this question in type A. They constructed a finite presentation  by adding the column generators, in the spirit of Kapur and Narendran in~\cite{KapurNarendran85}. They proved the convergence of this presentation using the combinatorial properties of Young tableaux.  But the  above question was still open for plactic monoids for the others types. 
\bigskip

In this work, we consider the plactic monoid  for type C constructed by Lecouvey in~\cite{Lecouvey02}. We construct a finite convergent presentation for this monoid, again by adding new generators.  The generating set of this presentation contains  the finite set of admissible columns introduced by Kashiwara and Nakashima in~\cite{KashiwaraNakashima94}. The right side of the relations of this presentation is the result of the Lecouvey's insertion of an admissible column into another one. In other words, we show that the right-hand sides of rewriting rules are symplectic tableaux consisting of at most two admissible columns. As a consequence,  we deduce that plactic monoids for type C satisfy  some homological finiteness properties.
\bigskip

The confluence of our  presentation is proved  using the unique normal form property and not by studying the confluence of the critical branchings. This method did not allow us to construct a coherent presentation  of the monoid $\mathbf{P_{n}}(C)$. Such a presentation extends the notion of a presentation of the monoid by homotopy generators taking into account the relations among the relations.  An interesting work would be to construct coherent presentations for  the  monoid $\mathbf{P_{n}}(C)$ which allow to describe the notion of an action of this monoid on categories, see~\cite{GaussentGuiraudMalbos14}. A coherent  presentation of a monoid is a first step to a  polygraphic resolution of the monoid, that is, a categorical cofibrant replacement of the monoid which can be  used to compute its  homological invariants~\cite{GuiraudMalbos12advances}. In~\cite{Lopatkin14}, Lopatkin  constructed Anick's resolution for the  monoid $\mathbf{P_{n}}(A)$ starting with a finite convergent presentation. Our  finite convergent presentation of  the  monoid $\mathbf{P_{n}}(C)$ should allow us to compute a polygraphic resolution of it which is a generalisation of  Anick's resolution. 
\bigskip

While submitting this paper, we came across the work of Cain, Gray and  Malheiro~\cite{CainGrayMalheiro14}. They construct by a  different method,  similar finite convergent presentations for plactic monoids of type B, C and D. They use Lecouvey's presentations of plactic monoids whereas we use Lecouvey's insertion algorithm. For type A, using Schensted's column insertion  we can insert a column $V$ into a column $U$ and during this insertion  either we add boxes  at the bottom of the column $U$ filled by the elements of $V$  or  the elements of  the column $V$ bump some boxes of $U$ into a new column. Thus we have directly that the result is a tableau consisting of at most two columns where the right one contains fewer elements than $U$. Note that it is more difficult to prove the later result using the Knuth presentation. For type C,  using Lecouvey's insertion we generalise this construction  and we prove  the  same results in Lemma~\ref{First Lemma} and Lemma~\ref{Second Lemma} for admissible columns, which is  in some sense more natural and more combinatorial than the other method.

In~\cite{Lecouvey03}, Lecouvey gave presentations for plactic monoids of type B and D and generalized the notion of admissible column to these types. He also introduced the notion of  \emph{orthogonal tableaux}~{\cite[Section~3]{Lecouvey03}}. Let $\mathcal{B}_{n}$ and $\mathcal{D}_{n}$  be respectively the alphabets corresponding to type B and D. Using the same insertion's algorithm  described in Sections~\ref{Insertioncolumn} and~\ref{insertiontableau}, Lecouvey showed that for any word $w$ in  the free monoids $\mathcal{B}_{n}^{\ast}$ and $\mathcal{D}_{n}^{\ast}$, one can compute a unique ortogonal tableau $P(w)$ which its reading is equal to $w$ in the corresponding plactic monoid.

Using the same strategy as in this paper, one can construct finite convergent presentations of plactic monoids for type B and D by introducing admissible column generators. The rewriting system rewrites two admissible columns that do not form an orthogonal tableau to their corresponding orthogonal tableau form. Since Kashiwara's theory of crystal graphs exists for type B and D, one can show that  Lemmas~\ref{First Lemma} and~\ref{Second Lemma} are also true for these types. Hence by this approach,  we should obtain the same result as Theorem~\ref{maintheorem} for plactic monoids of type B and D.
\bigskip

The paper is organised as follows.  We first recall in Section~\ref{preliminaries} the notion of $2$-polygraphs which corresponds to a presentation of a monoid by a rewriting system, that is a presentation by generators and oriented relations. After that, we present  some properties of crystal  graphs and Young diagram. In Section~\ref{Plactic monoid},  we present the definitions and  some properties of admissible columns and  symplectic tableaux. We describe the column insertion algorithm for type C introduced by Lecouvey in~\cite{Lecouvey02} and  a definition of the plactic monoid of type C. In Section~\ref{Convergent presentation}, we give a finite and convergent presentation of the plactic monoid for type C using admissible column generators.

\section{Preliminaries}
\label{preliminaries}

\subsection{ Rewriting properties of $2$-polygraphs}
 
We give some rewriting properties of the presentations of monoids. These presentations are studied in terms of polygraphs in~\cite{GuiraudMalbos14}.  A \emph{ $2$-polygraph } is a triple $\Sigma=(\Sigma_{0},\Sigma_{1},\Sigma_{2})$ made of an oriented graph
 \[
\xymatrix{
\Sigma_0 
& 
\Sigma_1
\ar@<+0.5ex>[l]^{t_0}
\ar@<-0.5ex>[l]_{s_0}
}
\]
where $\Sigma_{0}$ and $\Sigma_{1}$ are respectively the sets of  objects, or generating $0$-cells and of arrows, or generating  $1$-cells and $s_{0}$, $t_{0}$ denote the source and target  maps.  The set $\Sigma_{2}$ is a globular extension of the free category $\Sigma_{1}^{\ast}$, that is,  a set of $2$-cells equipped with source and target maps $s_{1},t_{1}: \Sigma_{2}\to \Sigma_{1}^{\ast}$ and relating parallel $1$-cells
 \[
 \xymatrix@C=2em{
 & x
 	\ar@/^3ex/ [rr] ^-{s_{1}(\alpha)} ^{}="src"
 	\ar@/_3ex/ [rr] _-{t_{1}(\alpha)} ^{}="tgt"
 	\ar@2 "src"!<0pt,-10pt>;"tgt"!<0pt,10pt> ^-{\alpha}
 && y
 & 
 }
\]
such that $s_{0}s_{1}(\alpha) = s_{0}t_{1}(\alpha)$ and $t_{0}s_{1}(\alpha) = t_{0}t_{1}(\alpha)$, where $s_{1}(\alpha),t_{1}(\alpha)\in \Sigma_{1}^{\ast}$. In our case,  we deal with monoids, that is, categories with only one $0$-cell, so that the set $\Sigma_{0}$ contains only one $0$-cell. In the sequel, the set $\Sigma_{0}$ is omitted and a $2$-polygraph is denoted by $\Sigma = (\Sigma_{1},\Sigma_{2})$.

A monoid $M$ is presented by a $2$-polygraph $\Sigma$ if $M$ is isomorphic to the quotient of the free monoid $\Sigma_{1}^{\ast}$ by the congruence generated by $\Sigma_{2}$. Then the generating $1$-cells are  the generators of $M$ and the generating $2$-cells correspond to the relations of $M$. Note that we will also say words for the $1$-cells of $\Sigma_{1}^{\ast}$ in a case of monoid. Denote by $l(w)$ the length of a word $w$ on $\Sigma_{1}^{\ast}$.

 A $2$-polygraph $\Sigma$ is \emph{finite} if $\Sigma_{0}$, $\Sigma_{1}$ and $\Sigma_{2}$ are finite. For two words $u$ and $v$ in $\Sigma_{1}^{\ast}$,  we write $u\dfl v$ for  a $2$-cell  in $\Sigma_{2}$.  A \emph{ rewriting step } of $\Sigma$ is a $2$-cell in $\Sigma_{2}$ with  shape 
 \[
 \xymatrix@C=2em{
 x
 	\ar [r] ^-{w} 
 & y 
 	\ar@/^3ex/ [rr] ^-{u} ^{}="src"
 	\ar@/_3ex/ [rr] _-{v} ^{}="tgt"
 	\ar@2 "src"!<0pt,-10pt>;"tgt"!<0pt,10pt> ^-{\phi}
 && z 	\ar [r] ^-{w'}
 & t
 }
\]
 where $\phi$ is a $2$-cell in $\Sigma_{2}$ and $w$ and $w'$ are words of $\Sigma_{1}^{*}$. A \emph{rewriting sequence} of $\Sigma$ is a finite or infinite sequence of rewriting steps. We say that $u$ rewrites into $v$ if $\Sigma$ has a  nonempty rewriting sequence from $u$ to $v$. A word of $\Sigma_{1}^{\ast}$ is a \emph{ normal form } if $\Sigma$ has no rewriting step with source $u$. A  normal form  of $u$ is a word $v$ of $\Sigma_{1}^{\ast}$ that is a normal form and such that $u$ rewrites into $v$. We say that $\Sigma$ \emph{ terminates } if it has no infinite rewriting sequences. We say that $\Sigma$ is \emph{ confluent }  if for any words $u$, $u'$ and $u''$ of $\Sigma_{1}^{\ast}$, such that $u$ rewrites into $u'$ and $u''$, there exists a word $v$ in $\Sigma_{1}^{\ast}$ such that $u'$ and $u''$ rewrite into $v$, that is, we have the following diagram 
\[\xymatrix @C=1.5em @R=0.5em {
& {u'}
  \ar@2{.>}[dr] ^-{}
\\
{u}
	\ar@2[ur] ^-{}
	\ar@2[dr] _-{}
&&{v}	
\\
& {u''}
\ar@2{.>}[ur] _-{}
}
\]
We say that $\Sigma$ is \emph{convergent} if it terminates and it is confluent. Note that a terminating $2$-polygraph is convergent if every word admits a unique normal form. 

Two $2$-polygraphs are \emph{ Tietze-equivalent} if they present the same monoid. Two finite $2$-polygraphs are Tietze-equivalent if, and only if, they are related by a finite sequence of elementary Tietze transformations. That is, one of the following transformations:
\begin{itemize}
\item adjunction or elimination of a $1$-cell $x$ and of a $2$-cell $\alpha: u\dfl x$, where $u$ is a $1$-cell of $(\Sigma_{1}\setminus\ens{x})^{\ast}$,
\item adjunction or elimination of a $2$-cell $\alpha: u\dfl v$ such that $u$ and $v$ are related by a nonoriented sequence of $2$-cells all in $\Sigma_{2}\setminus\ens{\alpha}.$ 
\end{itemize}

\subsection{ Crystal graphs}

Consider the following data. Let $\mathfrak{g}$ be a semisimple  Lie algebra. Let $P$ be the weight lattice for $\mathfrak{g}$ and  let $P^{\ast}= \mathrm{Hom}_{\mathbb{Z}}(P,\mathbb{Z})$. Let $\{\alpha_{i}\}_{i\in I}$ be the simple roots of $\mathfrak{g}$ and $\{h_{i}\}_{i\in I}$ the corresponding coroots. The two lattices $P$ and $P^{\ast}$ are free $\mathbb{Z}$-modules of rank $\sharp I$, see~\cite{BourbakiLie4-6}.  Let $\langle\cdot,\cdot\rangle: P^{\ast}\times P \to \mathbb{Z}$ be the canonical pairing. 
\medskip

A \emph{crystal} is a set $B$ endowed with applications
\[
\begin{array}{rl}
\textrm{w}t&:B\to P,\\
\epsilon_{i}&:B\to \mathbb{Z}\cup \{-\infty\},\\
\varphi_{i}&:B\to \mathbb{Z}\cup \{-\infty\},\\
\widetilde{e}_{i}&: B\to B\cup \{0\},\\
\widetilde{f}_{i}&:B\to B\cup \{0\}.
\end{array}
\]
satisfying  the following properties :
\begin{itemize}
\item $\varphi_{i}(b) = \epsilon_{i}(b)+ \langle h_{i},\textrm{w}t(b)\rangle$, for any $i$.
\item If $b\in B$ satisfies $\widetilde{e}_{i}(b)\neq 0$, then $\epsilon_{i}(\widetilde{e}_{i}(b))=\epsilon_{i}(b)-1,  \varphi_{i}(\widetilde{e}_{i}(b))=\varphi_{i}(b)+1$ and $ \textrm{w}t(\widetilde{e}_{i}(b)) = \textrm{w}t(b)+\alpha_{i}$.

\item If $b\in B$ satisfies $\widetilde{f}_{i}(b)\neq 0$, then $\epsilon_{i}(\widetilde{f}_{i}(b))=\epsilon_{i}(b)+1,  \varphi_{i}(\widetilde{f}_{i}(b))=\varphi_{i}(b)-1$ and $\textrm{w}t(\widetilde{f}_{i}(b)) = \textrm{w}t(b)-\alpha_{i}$.

\item For $b_{1},b_{2}\in B$, $b_{2}=\widetilde{f}_{i}(b_{1})$ if and only if $b_{1}=\widetilde{e}_{i}(b_{2})$.
\item If $\varphi_{i}(b)=-\infty$, then $\widetilde{e}_{i}(b) = \widetilde{f}_{i}(b) = 0$.
\end{itemize}
\medskip

The tensor product of two crystals $B_{1}$ and $B_{2}$ is defined by $$B_{1}\otimes B_{2} =\{b_{1}\otimes b_{2}\;\big|\; b_{1}\in~B_{1}, b_{2}\in B_{2}\}.$$
The set $B_{1}\otimes B_{2}$ is endowed with a structure of crystal by defining the action of $\widetilde{e}_{i}$ and $\widetilde{f}_{i}$ on the tensor product  by 
\[\widetilde{f}_{i}(b_{1}\otimes b_{2}) =\left\{
\begin{array}{lll}
\widetilde{f}_{i}(b_{1})\otimes b_{2} && \mbox{ if } \varphi_{i}(b_{1})>\varepsilon_{i}(b_{2})\\
b_{1}\otimes \widetilde{f}_{i}(b_{2}) && \mbox{ if } \varphi_{i}(b_{1})\leq \varepsilon_{i}(b_{2})
\end{array}
\right.\]
\[\widetilde{e}_{i}(b_{1}\otimes b_{2}) =\left\{
\begin{array}{lll}
b_{1}\otimes \widetilde{e}_{i}(b_{2}) && \mbox{ if } \varphi_{i}(b_{1})<\varepsilon_{i}(b_{2})\\
\widetilde{e}_{i}(b_{1})\otimes b_{2} && \mbox{ if } \varphi_{i}(b_{1})\geq \varepsilon_{i}(b_{2})
\end{array}
\right.\]
where $\varepsilon_{i}(b_{1}) = \text{max}\{k \;\big|\; \widetilde{e}_{i}^{k}(b_{1})\neq 0 \}$ and  $\varphi_{i}(b_{1}) = \text{max}\{k \;\big|\; \widetilde{f}_{i}^{k}(b_{1})\neq 0 \}$.
\medskip

\emph{Crystal graphs} are oriented  graphs with labeled arrows. The set of vertices is $B$  and an arrow $a\overset{i}{\rightarrow} b$ means that $\widetilde{f}_{i}(a) = b$ and $\widetilde{e}_{i}(b) = a$.
\bigskip

The \emph{symplectic Lie algebra} $\mathfrak{sp}_{2n}$ is the Lie algebra  of $2n$ by $2n$ matrices $A$, for $n>0$,  that satisfy  
$$\Omega A + A^\mathrm{T} \Omega = 0,$$ 
where $A^\mathrm{T}$ is the transpose of $A$ and $ \Omega = \begin{pmatrix} 0 & I_n \\ -I_n & 0 \\ \end{pmatrix}$. 

This Lie algebra is a semisimple Lie algebra of  type C and  we denote by $\Lambda_{i}$, for $i=1,\ldots,n$, its fundamental weights, see~\cite{BourbakiLie4-6}. In this case, $P=\underset{i}{\bigoplus}\mathbb{Z}\Lambda_{i}$.

Let $V_{n}=\mathbb{C}^{2n}$ be the vector representation of $\mathfrak{sp}_{2n}$, this representation is of dimension $2n$ and we index a basis of $V_{n}$ by the set 
\[
\mathcal{C}_{n} = \{1,2,\ldots, n,\overline{n},\ldots, \overline{1}\},
\]
totally ordered by  $1 < 2 < \ldots <n<\overline{n}<\ldots <\overline{1}$. Denote by $\mathcal{C}_{n}^{\ast}$ the free monoid over $\mathcal{C}_{n}$. 

Note that every representation of  the  Lie algebra $\mathfrak{sp}_{2n}$ admits a crystal graph. Recall that the crystal graph of  the vector representation $V_{n}$ is :
\[
1\overset{1}{\rightarrow}2\overset{2}{\rightarrow}\ldots \rightarrow n-1\overset{n-1}{\rightarrow}n\overset{n}{\rightarrow}\overline{n}\overset{n-1}{\rightarrow}\overline{n-1}\overset{n-2}{\rightarrow}\ldots \rightarrow\overline{2}\overset{1}\rightarrow\overline{1}.
\]

In~\cite{KashiwaraNakashima94}, Kashiwara and Nakashima  showed that the monoid $\mathcal{C}_{n}^{\ast}$ is a crystal and described a process to compute the action of the crystal operators $\widetilde{e}_{i}$ and $\widetilde{f}_{i}$ on a word $w$ of the monoid $\mathcal{C}_{n}^{\ast}$, for a fixed $i$. First, one considers the word $w_{i}$ obtained by deleting all symbols other that $i, i+1, \overline{i+1}$ and $\overline{i}$ from $w$. One identifies the letters $i$ and $\overline{i+1}$ by the symbol $+$ and the letters $i+1$ and $\overline{i}$ by the symbol $-$. Secondly, we remove the subwords of length $2$ in $w_{i}$ which correspond to the symbol $+-$, \ie, we remove adjacent letters $(i,i+1), (i,\overline{i}), (\overline{i+1}, i+1)$  and $(\overline{i+1},\overline{i})$. Then we obtain a new subword of $w$. The second step of the  process is repeated until it is impossible to remove more letters. Let $r$ and $s$ be respectively the number of letters corresponding to  the symbols $-$ and $+$  in the final subword. 

\begin{itemize}
\item If $r>0$ then $\widetilde{e}_{i}(w)$ is obtained by replacing in $w$ the rightmost element with the symbol $-$ of the final subword, by its corresponding element with  the symbol $+$, \ie, $i+1$ is transformed into $i$ or $\overline{i}$ into $\overline{i+1}$ or for $i=n$, $\overline{n}$ into $n$, and the others elements of $w$ stay unchanged. If $r = 0$, then $\widetilde{e}_{i}(w) = 0$.
\item If $s>0$ then $\widetilde{f}_{i}(w)$ is obtained by replacing in $w$ the leftmost element with  the symbol $+$ of the final subword, by its corresponding element with the  symbol $-$, \ie, $i$ is transformed into $i+1$ or $\overline{i+1}$ into $\overline{i}$ or for $i=n$, $n$ into $\overline{n}$, and the others elements of $w$ stay unchanged. If $s = 0$, then $\widetilde{f}_{i}(w) = 0$.
\end{itemize}

\begin{example}
Consider the word $w = \overline{3}32313\overline{3}233\overline{3}1$. For $i=2$, we have $w_{i} = \overline{3}3233\overline{3}233\overline{3}$. After deleting subwords corresponding to  $+-$, the first subword of $w_{i}$ is $3\overline{3}3\overline{3}$. After repeating this process, the second subword is $3\overline{3}$. We cannot remove new elements from the last subword, then $r=s=1$. Finally, we obtain :
\[\widetilde{e}_{2}(w) = \overline{3}32312\overline{3}233\overline{3}1 \text{ }\text{ and }\text{ }\widetilde{f}_{2}(w) = \overline{3}32313\overline{3}233\overline{2}1.\]
\end{example}
\bigskip

Now, we consider tensor products of the vector representation $V_{n}^{\otimes l}$, for any $l$ and the infinite dimensional representation $\underset{l}{\bigoplus}V_{n}^{\otimes l}$.  The crystal graphs of these representations are denoted by $G_{n,l}$ and $G_{n}$, respectively.   Note that each vertex $x_{1}\otimes x_{2}\otimes \ldots \otimes x_{l}$ of the crystal graph $V_{n}^{\otimes l}$ is identified with the word $x_{1}x_{2}\ldots x_{l}$ in the monoid $\mathcal{C}_{n}^{\ast}$. In other words,  the vertices of $G_{n}$ are indexed by the words of $\mathcal{C}_{n}^{\ast}$ and those of $G_{n,l}$ by the words of length $l$. 
 
In addition, the crystal graph $G_{n,l}$ can be decomposed into connected components. They correspond to the crystal graphs of the irreducible representations  occurring in the decomposition of $V_{n}^{\otimes l}$. If $w$ is a vertex of $G_{n,l}$, the connected component  of $G_{n,l}$ containing $w$ is denoted by $B(w)$. In each connected component,  there exists a unique vertex $w^{0}$ which satisfy the following property:
\[\widetilde{e}_{i}(w^{0}) = 0, \text{ } \text{ for } i = 1,\ldots, n.\]
This vertex is called the \emph{ vertex of highest  weight}, and its weight is
\[
\textrm{wt}(w^{0}) = d_{n}\Lambda_{n} + \overset{n-1}{\underset{i=1}{\sum}}(d_{i} - d_{i+1})\Lambda_{i},
\]
where $d_{i}$ is the number of letters $i$ in $w^{0}$ minus the number of letters $\overline{i}$. Two connected components are isomorphic if
there is a weight-preserving labeled digraph isomorphism from one to the other.  Note that this isomorphism is unique.

\begin{example}
 For $n=2$,  the crystal $B(11)$ is presented by : 
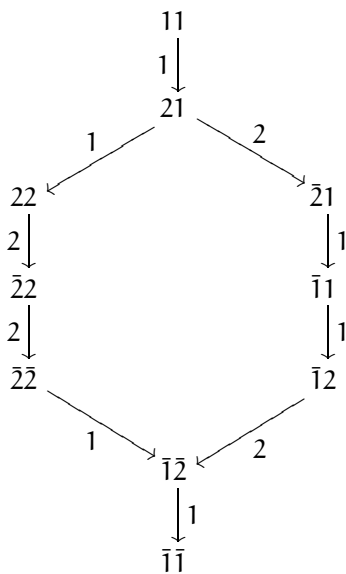
\[\small
\xymatrix@C=4em @R=0.7cm{
&&11
\ar@<+0.4ex>[d] _-{1}\\
&& 21
 	\ar@<+0.4ex>[dl] _-{1}
 	\ar@<+0.4ex>[dr]^-{2}
 \\
&22
	\ar@<+0.4ex>[d]_-{2}
&& \bar{2}1
\ar@<+0.4ex>[d]^-{1}
\\
&\bar{2}2
\ar@<+0.4ex>[d]_-{2}
&& \bar{1}1
\ar@<+0.4ex>[d]^-{1}
\\
&\bar{2}\bar{2}
\ar@<+0.4ex>[dr]_-{1}
&&\bar{1}2
\ar@<+0.4ex>[dl]^-{2}
\\
&&\bar{1}\bar{2}
\ar@<+0.4ex>[d]^-{1}
\\
&&\bar{1}\bar{1}
 } 
\]
where the vertices are labeled by words. In this case, the vertex of highest weight is $11$ and its weight is $2\Lambda_{1}$.
\end{example}

\begin{lemma}[\cite{KashiwaraNakashima94}]
\label{highestweight}
For any words $u$ and $v$ in $\mathcal{C}_{n}^{\ast}$, the word $uv$ is  a vertex of highest weight of a connected component of $G_{n}$ if, and only if, $u$ is a vertex of highest weight and $\varepsilon_{i}(v)\leq \varphi_{i}(u)$ for any $i= 1,\ldots, n$. 
\end{lemma}

For more details about crystal graphs, the reader is referred to~\cite{Kashiwara91,Kashiwara94,KashiwaraNakashima94}.

\subsection{Young diagram}
A \emph{Young diagram} is a collection of boxes in left-justified rows, where each row has the same or shorter length than the one above it. 

Let $\lambda = \overset{n}{\underset{i=1}{\sum}}\lambda_{i}\Lambda_{i}$ be the highest weight of an irreducible representation of $\mathfrak{sp}_{2n}$,  with $\lambda_{i}\geq 0$. Note  that $\lambda$ corresponds to the Young diagram as follows. For $\lambda$, we associate  the Young diagram  $Y(\lambda)$  containing $\lambda_{i}$ columns of height $i$. We say that this Young diagram has shape $\lambda$ and  the number of its boxes  is equal to $\left| \lambda\right|= \overset{n}{\underset{i=1}{\sum}}\lambda_{i}i.$ 

\begin{example}
The Young diagram $Y(2\Lambda_{1} + 3\Lambda_{2}+ \Lambda_{3})$ is
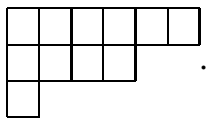
\[\begin{tabular}[c]{|l|lllll}\hline
&\multicolumn{1}{|l|}{}&\multicolumn{1}{|l|}{}&\multicolumn{1}{|l|}{}&\multicolumn{1}{|l|}{}&\multicolumn{1}{|l|}{}\\\cline{1-6}
&\multicolumn{1}{|l|}{}&\multicolumn{1}{|l|}{}&\multicolumn{1}{|l|}{}&\\\cline{1-4}
&\\\cline{1-1}
\end{tabular}.
\]
\end{example}

Denote by $B(\lambda)$ the  connected component of the crystal graph such  that its vertex of highest weight has weight $\lambda$.

\section{Plactic monoid for type C}
\label{Plactic monoid}

\subsection{Symplectic Tableaux}
\label{symplectic tableaux}

A \emph{ column } for type C is a Young diagram $U$ consisting of  one column  filled by letters of $\mathcal{C}_{n}$ strictly increasing from top to bottom. We call the \emph{reading} of  a column $U$ the word $\textrm{w}(U)$  obtained by reading the letters of  $U$ from top to bottom.  The height of a column $U$ is the number of letters in $U$ and denoted by $h(U)$. A word $w$ is a column word if there exists a column $U$ such that $w = \textrm{w}(U)$. 

For example the Young diagram 
\[U=
\begin{tabular}
[c]{|l|}\hline
$\mathtt{1}$ \\ \hline
$\mathtt{2}$\\ \hline
$\mathtt{3}$\\ \hline
$\mathtt{\overline{6}}$\\ \hline
$\mathtt{\overline{5}}$\\ \hline
\end{tabular}
\]
is a column. Its reading is $\textrm{w}(U)= 123\overline{6}\,\overline{5}.$

In~\cite{KashiwaraNakashima94}, Kashiwara and Nakashima  introduced the notion of admissible column. Let $\textrm{w}(U) = x_{1}\ldots x_{h(U)}$ be the reading of a column $U$.  For us,  the column $C$ is \emph{ admissible } if for $m=1,\ldots , h(U)$,  the number $N(m)$ of letters $x$ in $U$ such that $x\leq m$ or $x\geq \overline{m}$ satisfies $N(m)\leq m$.

Let $U$ be a column and $I=\{x_{1}>\ldots > x_{r}\}$ be the set of unbarred letters such that $x_{i}, \overline{x}_{i}\in U$ , for $i=1,\ldots, r$.  The column $U$ can be \emph{split} if there exists a set of unbarred letters $J = \{y_{1}>\ldots >y_{r} \}$ containing $r$ elements of $\mathcal{C}_{n}$ such that : 
\begin{itemize}
\item $y_{1}$ is the greatest letter of $\mathcal{C}_{n}$ satifying $y_{1} < x_{1}, y_{1}\notin U$ and $\overline{y_{1}}\notin U$,
\item for $i=2,\ldots, r,$ $y_{i}$ is the greatest letter of $\mathcal{C}_{n}$ such that $y_{i}< \text{min}(y_{i-1}, x_{i})$, $y_{i}\notin U$ and 
$\overline{y_{i}}\notin U$.
\end{itemize}
Denote by $rU$ the column obtained by changing in $U$, $\overline{x}_{i}$ into $\overline{y}_{i}$ for each letter $x_{i}$ in the set $I$ up to reordering. Denote by $lU$ the column obtained by changing in $U$, $x_{i}$ into $y_{i}$ for each letter $x_{i}$ in the set $I$ up to reordering.

\begin{proposition}[{\cite[Section~4]{Sheats99}}]
A column $U$ is admissible if,  and only if,  it can be split.
\end{proposition}

\begin{example}
Let $\textrm{w}(U)= 2568\overline{8}\,\overline{5}\,\overline{2}$ be the  reading of a column $U$. Then 
\[
\begin{array}{rl}
I=\{ 8>5>2\}&,\quad J=\{7>4>1\},\\
\textrm{w}(rU)= 2568\overline{7}\,\overline{4}\,\overline{1} \quad&\text{ and } \quad\textrm{w}(lU) = 1467\overline{8}\,\overline{5}\,\overline{2}.
\end{array}
\]
 The column $U$ can be split, so that it is an admissible column.
\end{example}

\begin{example}
Let $\textrm{w}(U')= 2346\overline{6}\,\overline{3}\,\overline{2}$ be the  reading of a column $U'$. Then 
$$I=\{6, 3, 2\},\quad y_{1}=5,\quad  y_{2} = 1$$
and we cannot find an element $y_{3}$ of $\mathcal{C}_{n}$ such that $y_{3}<1$. Thus $U'$ cannot be split. 
\end{example}

Using admissible columns, one can construct a tableau  whose columns are admissible with an additional property on them. This tableau is called the symplectic tableau. We will recall its definition in our context. Let $U_{1},\ldots, U_{r}$ be the $r$ columns from left to right of a Young  diagram $T$, then $T$ is denoted by $T=U_{1}\ldots U_{r}$.
\medskip

Let $U_{1}$ and $U_{2}$ be two admissible columns. Consider the following notation :
\begin{itemize}
\item $U_{1}\leq U_{2}$ if $h(U_{1}) \geqslant h(U_{2})$ and the rows of the tableau $U_{1}U_{2}$ are weakly increasing from left to right.
\item $U_{1}\preceq U_{2}$ if $rU_{1}\leq lU_{2}$.
\end{itemize}
\medskip

Consider a tableau $T=U_{1}\ldots U_{r}$, with admissible column $U_{i}$, for $i=1,\ldots, r$. The tableau $T$ is a \emph{ symplectic tableau} if $U_{i}\preceq U_{i+1}$ for $i=1,\ldots, r-1.$  The reading of the symplectic tableau $T$ is  the word $\textrm{w}(T)$ obtained by reading the columns of $T$  from right to left, that is
$$\textrm{w}(T) = \textrm{w}(U_{r})\textrm{w}(U_{r-1})\ldots \textrm{w}(U_{1}).$$

\begin{example}
Let us consider the tableau 
\[T=
\begin{tabular}
[c]{|l|l|l}\hline
$\mathtt{1}$ & $\mathtt{2}$ & \multicolumn{1}{|l|}{$\mathtt{3}$}\\ \hline
$\mathtt{2}$ & $\mathtt{\overline{3}}$\\ \cline{1-2}
$\mathtt{3}$ & $\mathtt{\overline{2}}$\\ \cline{1-2}
\end{tabular}
\]
$T$ is a symplectic tableau. Indeed,
\begin{itemize}
\item $\textrm{w}(U_{1}) = 123,\quad  I_{U_{1}} = J_{U_{1}} = \varnothing$\quad  and\quad  $\textrm{w}(rU_{1}) = \textrm{w}(lU_{1})= 123$.
\item $\textrm{w}(U_{2}) = 2\overline{3}\,\overline{2},\quad  I_{U_{2}} = \{2\},\quad  J_{U_{2}} = \{1\},\quad  \textrm{w}(rU_{2}) = 2\overline{3}\,\overline{1}$\quad  and\quad  $\textrm{w}(lU_{2})~=~1\overline{3}\,\overline{2}.$
\item $\textrm{w}(U_{3}) = 3,\quad  I_{U_{3}} = J_{U_{3}} = \varnothing$\quad  and\quad  $\textrm{w}(rU_{3}) = \textrm{w}(lU_{3}) = 3.$
\end{itemize}
The columns $U_{1}, U_{2}$ and $U_{3}$ can be split, so they are admissible columns. We have $U_{1}~\preceq~U_{2}~\preceq~U_{3}$,  so $T$ is a symplectic tableau and $\textrm{w}(T) = 32\overline{3}\,\overline{2}123.$
\end{example}

\begin{remark}
\label{Highestweightremark}
Let $\lambda = \overset{n}{\underset{i=1}{\sum}}\lambda_{i}\Lambda_{i}$ be a weight with  $\lambda_{i}\geqslant 0$. By Theorem~4.5.1 
in~\cite{KashiwaraNakashima94}, $B(\lambda)$ coincides with the set of symplectic tableaux of shape $\lambda$. More precisely, the readings of these tableaux are the vertices of a connected component of $G_{n,\left| \lambda \right|}$ isomorphic to $B(\lambda)$. The highest weight vertex of this component is the reading of the tableau of shape $\lambda$ filled with $1$ on the $1$st row, $2$ on the $2$nd row, ... , and $n$ on the $n$th row. In particular, the reading of the highest weight vertex of a connected component containing admissible columns of height $p$ is $12\ldots p$.
\end{remark}

\subsection{Definition of the plactic monoid for type C}

Recall that for type A, we consider the ordered alphabet $\mathcal{A}_{n} = \{1<2<\ldots <n\}$. The plactic monoid $\mathbf{P_{n}}(A)$ of type A is presented by the quotient of $\mathcal{A}_{n}^{*}$ by the congruence generated by the Knuth relations~(\ref{Knuthrelations}). This presentation is called the \emph{Knuth presentation}. Note that the Knuth presentation can be also described using Kashiwara's theory of crystal graphs, see~\cite{LascouxLeclercThibon95}.
\bigskip

Let us define the plactic monoid for type C. Let $u$ and $v$ be two words in $\mathcal{C}_{n}^{\ast}$. One can define a relation $\sim$ on the free monoid $\mathcal{C}_{n}^{\ast}$ by : $u\sim v$ if,  and only if, $B(u)$ and $B(v)$ are isomorphic and $u$ and $v$ have the same position in the isomorphic connected component $B(u)$ and $B(v)$ of the crystal $G_{n}$. In other words, $u \sim v$ if and only if there exist $i_{1},\ldots, i_{r}$ such that $u=\widetilde{f}_{i_{i}}\cdot\cdot\cdot \widetilde{f}_{i_{r}}(u^{0})$ and $v=\widetilde{f}_{i_{i}}\cdot\cdot\cdot \widetilde{f}_{i_{r}}(v^{0})$, where $u^{0}$ and $v^{0}$ are the vertices of highest weight of $B(u)$ and $B(v)$.

\begin{proposition}[{\cite[Proposition~3.1.2]{Lecouvey02}}]
Every word $w$ in $\mathcal{C}_{n}^{\ast}$ admits a unique symplectic tableau $T$ such that $w\sim \textrm{w}(T)$. 
\end{proposition}

The unique symplectic tableau $T$ such that $w\sim \textrm{w}(T)$ is denoted by $P(w)$. The quotient $\mathbf{P_n}(C):=~\mathcal{C}_{n}^{\ast}/\sim$ is called the \emph{ plactic  monoid for type C } or  the  \emph { symplectic  plactic   monoid}.
\medskip

Furthermore, the plactic monoid for type C can be presented by generators and relations. Consider the congruence~$\equiv$ generated by the following  families of relations on $\mathcal{C}_{n}^{\ast}$:

$(R_{1}):\left\{
\begin{array}{lll}
yzx\equiv yxz &&\mbox{ for }x\leq y <z \mbox{ with } z\neq\overline{x}\\
xzy\equiv zxy &&\mbox{ for }x<y\leq z \mbox{ with } z\neq\overline{x}
\end{array}
\right.$

$(R_{2}):\left\{
\begin{array}{lll}
y\overline{(x-1)}(x-1)\equiv yx\overline{x} &&\mbox{ for } 1<x \leq n \mbox{ and } x\leq y\leq\overline{x}\\
x\overline{x}y\equiv \overline{(x-1)}(x-1)y &&\mbox{ for } 1<x \leq n \mbox{ and } x\leq y\leq\overline{x}
\end{array}
\right.$

$(R_{3}):$ let $w$ be a nonadmissible column word such that each strict factor of it is an admissible column word. Let $z$ be the lowest unbarred letter such that $z,\overline{z}\in w$ and $N(z) = z+1$. Then $w\equiv\widetilde{w}$, where $\widetilde{w}$ is the column word obtained by erasing $z$ and $\overline{z}$ from $w$.

\begin{remark}
The relations $(R_{1})$ contain the Knuth relations for type A. The relations $(R_{3})$ are called the \emph{contraction relations}.
\end{remark}

\begin{theorem}[{\cite[Theorem~3.2.8]{Lecouvey02}}]
\label{Cross section}
For any words $u$ and $v$ in $\mathcal{C}_{n}^{\ast}$, we have
\[ u\sim v \text{  }\text{ if and only if }\text{ } u\equiv v \text{ }\text{ if and only if  }\text{ } P(u) = P(v).\]
\end{theorem}

\subsection{A bumping algorithm for type C}

In~\cite{Schensted61}, Schensted introduces an  insertion algorithm (column insertion)  to compute a unique tableau $P(w)$ for a word $w$ over the alphabet $\mathcal{A}_{n} = \{1<\ldots <n\}$. The column insertion procedure inserts a letter $x$ into a tableau $T$ as follows. Let $y$ be the smallest element of the leftmost column of the tableau $T$ such that $y\geq x$. Then $x$ replaces $y$ in the leftmost column and $y$ is bumped into the next column where the process is repeated. This procedure terminates when the letter which is bumped is greater than all the elements of the next column. Then it is placed at the bottom of that column.  Hence  the tableau $P(w)$ can be computed by starting with the empty word, which is a valid tableau, and iteratively applying Schensted's algorithm.
\bigskip

In~\cite{Lecouvey02}, Lecouvey introduces an insertion scheme to compute the symplectic tableau $P(w)$ analogous to the Schensted's algorithm for type A.
We present in Sections~\ref{Insertioncolumn} and~\ref{insertiontableau} Lecouvey's algorithms and we refer the reader to~\cite{Lecouvey02} for more details.

Let denote  by $x\rightarrow T$ the insertion of a letter $x$ in a symplectic tableau $T$.

\subsubsection{Insertion of a letter in an admissible column}
\label{Insertioncolumn}

Consider a word $w = \textrm{w}(U)x$, where $x$ is a letter and $U$ is an admissible column of height $p$. We have three cases :
\begin{itemize}
\item If $w$ is the reading of an admissible column, then $x\rightarrow U$ is the column obtained by adding a box filled by the letter $x$ at the bottom of $U$. In this case, the highest weight vertex of $B(w)$ is equal to $1\ldots p(p+1).$
\item If $w$ is a nonadmissible column word such that each  strict factor of it is admissible, then $x\rightarrow U$ is the column of reading $\widetilde{w}$ obtained from $w$ by applying one relation of type $(R_{3})$, which is uniquely determined~\cite{Lecouvey02}. In this case, the highest weight vertex of $B(w)$ is equal to $1\ldots p\overline{p}$.
\item If $w$ is not a column word, then $x\rightarrow U$ is obtained by applying relations of type $(R_{1})$ or $(R_{2})$ to the  final subword of length 3 of $w$. On the resulting word, one continues by applying relations of type $(R_{1})$ or $(R_{2})$ to the maximal overlapping factor of length 3 to the left and this procedure is repeated until the first factor of length 3 has been operated. The  result is the reading of a symplectic tableau consisting of a column $U'$ of height $p$ and a column 
\begin{tabular}
[c]{|l|}\hline
$x'$\\ \hline
\end{tabular}, where $x'$ is an element of $\mathcal{C}_{n}$. Then
\[
x\rightarrow U = U'
\begin{tabular}
[c]{|l|}\hline
$x'$\\ \hline
\end{tabular}
= P(w).\]
In this case, the highest weight vertex of $B(w)$ is equal to $1\ldots p1.$
\end{itemize}

\begin{example}
Let us consider the following three examples.
\begin{enumerate}
\item Suppose $\textrm{w}(U) = 36\overline{6}\,\overline{4}$  and $x = \overline{3}$, then
\[\mathtt{\overline{3}}\rightarrow
\begin{tabular}
[c]{|l|}\hline
$\mathtt{3}$\\\hline
$\mathtt{6}$\\\hline
$\mathtt{\overline{6}}$\\\hline
$\mathtt{\overline{4}}$\\\hline
\end{tabular}
=
\begin{tabular}
[c]{|l|}\hline
$\mathtt{3}$\\\hline
$\mathtt{6}$\\\hline
$\mathtt{\overline{6}}$\\\hline
$\mathtt{\overline{4}}$\\\hline
$\mathtt{\overline{3}}$\\\hline
\end{tabular}
.\]
\item Suppose $\textrm{w}(U) = 14\overline{4}\,\overline{3}$ and $x=\overline{2}$, the word $14\overline{4}\,\overline{3}\,\overline{2}$ is a nonadmissible column word such that each strict subword of it is an admissible column word, then we obtain by applying relation of type $(R_{3})$,
\[\mathtt{\overline{2}}\rightarrow
\begin{tabular}
[c]{|l|}\hline
$\mathtt{1}$\\\hline
$\mathtt{4}$\\\hline
$\mathtt{\overline{4}}$\\\hline
$\mathtt{\overline{3}}$\\\hline
\end{tabular}
=
\begin{tabular}
[c]{|l|}\hline
$\mathtt{1}$\\\hline
$\mathtt{\overline{3}}$\\\hline
$\mathtt{\overline{2}}$\\\hline
\end{tabular}
.\]
\item Suppose $\textrm{w}(U) = 14\overline{4}\,\overline{3}$ and $x=2$, then  the word $14\overline{4}\,\overline{3}2$ is not a column word. By applying relations of type $(R_{1})$ or $(R_{2})$, we obtain:
\[14\mathbf{\overline{4}\,\overline{3}2}\equiv 1\mathbf{4\overline{4}2}\overline{3}\equiv \mathbf{142}\overline{4}\,\overline{3}\equiv 412\overline{4}\,\overline{3}.\]
Then
\[\mathtt{2}\rightarrow
\begin{tabular}
[c]{|l|}\hline
$\mathtt{1}$\\\hline
$\mathtt{4}$\\\hline
$\mathtt{\overline{4}}$\\\hline
$\mathtt{\overline{3}}$\\\hline
\end{tabular}
=
\begin{tabular}
[c]{|l|l}\hline
$\mathtt{1}$ & \multicolumn{1}{|l|}{$\mathtt{4}$}\\\hline
$\mathtt{2}$\\\cline{1-1}
$\mathtt{\overline{4}}$\\\cline{1-1}
$\mathtt{\overline{3}}$\\\cline{1-1}
\end{tabular}
.\]
\end{enumerate}
\end{example}

\subsubsection{Insertion of a letter in a symplectic tableau}
\label{insertiontableau}

Let $T=U_{1}\ldots U_{r}$ be a symplectic tableau with admissible column $U_{i}$, for $i=1,\ldots,r,$ and $x$ be a letter. We have three cases:
\begin{itemize}
\item If $\textrm{w}(U_{1})x$ is an admissible column word, then $x\rightarrow T$ is the tableau obtained by adding a box filled by  $x$ on the bottom of $U_{1}$.
\item If $\textrm{w}(U_{1})x$ is a nonadmissible column word such that each strict factor of it is an admissible column word. Let $\widetilde{\textrm{w}(U_{1})x} = y_{1}\ldots y_{s}$ be the admissible column word obtained from $\textrm{w}(U_{1})x$ by applying relation of type $(R_{3})$ and  $\widehat{T}=U_{2}\ldots U_{r}$ be the tableau obtained from $T$ after eliminating the leftmost column $U_{1}$. Then $x\rightarrow T$ is obtained by inserting successively the elements of  $\widetilde{\textrm{w}(U_{1})x}$ in the tableau $\widehat{T}$. That is,
\[
x\rightarrow T = y_{s}\rightarrow(y_{s-1}\rightarrow(\cdot\cdot\cdot y_{1}\rightarrow \widehat{T})).
\]
Moreover, the insertion of $y_{1},\ldots, y_{s}$ in $\widehat{T}$ does not cause a new contraction.
\item If $\mathrm{w}(U_{1})x$ is not a column word, then 
\[ x\rightarrow U_{1}=
\begin{tabular}
[c]{|l|l}\hline
$\mathtt{U_{1}'}$& \multicolumn{1}{|l|}{$\mathtt{y}$}\\\hline
\end{tabular},\]
where $U_{1}'$ is an admissible column of height $h(U_{1})$ and $y$ a letter. Then 
\[x\rightarrow T = U_{1}'(y\rightarrow U_{2}\ldots U_{r}),\]
that is, $x\rightarrow T$ is the juxtaposition of $U_{1}'$ with the tableau obtained by inserting $y$ in the tableau $U_{2}\ldots U_{r}.$
\end{itemize}

\begin{example}
Consider a symplectic tableau 
\[T_{1} =
\begin{tabular}
[c]{|l|l|l}\hline
$\mathtt{1}$ & $\mathtt{2}$&\multicolumn{1}{|l|}{$\mathtt{3}$}\\\hline
$\mathtt{2}$ & $\mathtt{\overline{3}}$\\\cline{1-2}
$\mathtt{3}$ & $\mathtt{\overline{2}}$\\\cline{1-2}
\end{tabular}
\]
and a letter $x = 1$. Let us compute $x\rightarrow T_{1}$. First, we begin inserting $x$ in the leftmost column $U_{1}$ of $T_{1}$. The word $1231$ is not a column word, then by applying at each step $(R_{1})$ or $(R_{2})$, we obtain :
\[1\mathbf{231}\equiv \mathbf{121}3\equiv 1123,\]
so
\[\mathtt{1}\rightarrow
\begin{tabular}
[c]{|l|}\hline
$\mathtt{1}$\\\hline
$\mathtt{2}$\\\hline
$\mathtt{3}$\\\hline
\end{tabular}
=
\begin{tabular}
[c]{|l|l}\hline
$\mathtt{1}$ & \multicolumn{1}{|l|}{$\mathtt{1}$}\\\hline
$\mathtt{2}$\\\cline{1-1}
$\mathtt{3}$\\\cline{1-1}
\end{tabular}
.\]
Then $1\rightarrow T_{1} = U_{1}'(1\rightarrow T_{1}'),$ where 
\[
U_{1}'=
\begin{tabular}
[c]{|l|}\hline
$\mathtt{1}$\\\hline
$\mathtt{2}$\\\hline
$\mathtt{3}$\\\hline
\end{tabular}
\text{ }\text{ and }\text{ }
T_{1}' =
\begin{tabular}
[c]{|l|l}\hline
$\mathtt{2}$ & \multicolumn{1}{|l|}{$\mathtt{3}$}\\\hline
$\mathtt{\overline{3}}$\\\cline{1-1}
$\mathtt{\overline{2}}$\\\cline{1-1}
\end{tabular}
.\]
Similarly, we have $2\mathbf{\overline{3}\,\overline{2}1}\equiv \mathbf{2\overline{3}1}\overline{2}\equiv 21\overline{3}\,\overline{2},$ then 
\[\mathtt{1}\rightarrow
\begin{tabular}
[c]{|l|}\hline
$\mathtt{2}$\\\hline
$\mathtt{\overline{3}}$\\\hline
$\mathtt{\overline{2}}$\\\hline
\end{tabular}
=
\begin{tabular}
[c]{|l|l}\hline
$\mathtt{1}$ & \multicolumn{1}{|l|}{$\mathtt{2}$}\\\hline
$\mathtt{\overline{3}}$\\\cline{1-1}
$\mathtt{\overline{2}}$\\\cline{1-1}
\end{tabular}
.\]
So $1\rightarrow T_{1} = U_{1}'U_{2}'\big(2\rightarrow 
\begin{tabular}
[c]{|l|}\hline
$\mathtt{3}$\\\hline
\end{tabular}\big)$, where
\[
U_{2}' = 
\begin{tabular}
[c]{|l|}\hline
$\mathtt{1}$\\\hline
$\mathtt{\overline{3}}$\\\hline
$\mathtt{\overline{2}}$\\\hline
\end{tabular}
.\]
Finally, we have $32\equiv 32$, then
\[ \mathtt{2}\rightarrow 
\begin{tabular}
[c]{|l|}\hline
$\mathtt{3}$\\\hline
\end{tabular}
=
\begin{tabular}
[c]{|l|l|}\hline
$\mathtt{2}$ &$\mathtt{3}$\\\hline
\end{tabular}
.\]
Hence,
\[
\mathtt{1}\rightarrow T_{1} =
\begin{tabular}
[c]{|l|l|ll}\hline
$\mathtt{1}$ & $\mathtt{1}$ &\multicolumn{1}{|l|}{$\mathtt{2}$} &\multicolumn{1}{|l|}{$\mathtt{3}$}\\\hline
$\mathtt{2}$ & $\mathtt{\overline{3}}$ \\\cline{1-2}
$\mathtt{3}$ & $\mathtt{\overline{2}}$\\\cline{1-2}
\end{tabular}
.\]
\end{example}

\begin{example}
Consider a symplectic tableau
\[ T_{2} =
\begin{tabular}
[c]{|l|l|ll}\hline
$\mathtt{1}$ & $\mathtt{1}$ &\multicolumn{1}{|l|}{$\mathtt{2}$} &\multicolumn{1}{|l|}{$\mathtt{3}$}\\\hline
$\mathtt{2}$ & $\mathtt{3}$ &\multicolumn{1}{|l|}{$\mathtt{3}$}\\\cline{1-3}
$\mathtt{3}$ & $\mathtt{\overline{3}}$\\\cline{1-2}
\end{tabular}
\]
and a letter $x= \overline{3}$. Let us compute $x\rightarrow T_{2}$. First, we begin inserting $x=\overline{3}$ in the leftmost column $U_{1}$ of $T_{2}$. The word $123\overline{3}$ is a nonadmissible column word, that each strict factor is an admissible column word, we have by applying $(R_{3}),$
\[ 123\overline{3}\equiv 12,\]
then
\[\widetilde{U_{1}}=
\begin{tabular}
[c]{|l|}\hline
$\mathtt{1}$\\\hline
$\mathtt{2}$\\\hline
\end{tabular}
\text{ }\text{ and }\text{ }
\widehat{T_{2}}=
\begin{tabular}
[c]{|l|ll}\hline
$\mathtt{1}$ & \multicolumn{1}{|l|}{$\mathtt{2}$} & \multicolumn{1}{|l|}{$\mathtt{3}$}\\\hline
$\mathtt{3}$ & \multicolumn{1}{|l|}{$\mathtt{3}$} \\\cline{1-2}
$\mathtt{\overline{3}}$\\\cline{1-1}
\end{tabular}
.\]
So we have to insert the elements of the column $\widetilde{U_{1}}$ in the tableau $\widehat{T_{2}}$.
 
\noindent First, one inserts 1 :
\[1\mathbf{3\overline{3}1}\equiv \mathbf{131}\overline{3}\equiv 113\overline{3},\]
then 
\[\mathtt{1}\rightarrow
\begin{tabular}
[c]{|l|}\hline
$\mathtt{1}$\\\hline
$\mathtt{3}$\\\hline
$\mathtt{\overline{3}}$\\\hline
\end{tabular}
=
\begin{tabular}
[c]{|l|l}\hline
$\mathtt{1}$ & \multicolumn{1}{|l|}{$\mathtt{1}$}\\\hline
$\mathtt{3}$\\\cline{1-1}
$\mathtt{\overline{3}}$\\\cline{1-1}
\end{tabular}
.\]
We have  $231\equiv 213,$ then 
\[\mathtt{1}\rightarrow
\begin{tabular}
[c]{|l|}\hline
$\mathtt{2}$\\\hline
$\mathtt{3}$\\\hline
\end{tabular}
=
\begin{tabular}
[c]{|l|l}\hline
$\mathtt{1}$ & \multicolumn{1}{|l|}{$\mathtt{2}$}\\\hline
$\mathtt{3}$\\\cline{1-1}
\end{tabular}
.\]
And 
\[\mathtt{2}\rightarrow
\begin{tabular}
[c]{|l|}\hline
$\mathtt{3}$\\\hline
\end{tabular}
=
\begin{tabular}
[c]{|l|l}\hline
$\mathtt{2}$ & \multicolumn{1}{|l|}{$\mathtt{3}$}\\\hline
\end{tabular}
.\]
Hence
\[\mathtt{1}\rightarrow \widehat{T_{2}} =
\begin{tabular}
[c]{|l|lll}\hline
$\mathtt{1}$ & \multicolumn{1}{|l|}{$\mathtt{1}$} & \multicolumn{1}{|l|}{$\mathtt{2}$} & \multicolumn{1}{|l|}{$\mathtt{3}$}\\\hline
$\mathtt{3}$ & \multicolumn{1}{|l|}{$\mathtt{3}$}\\\cline{1-2}
$\mathtt{\overline{3}}$\\\cline{1-1}
\end{tabular}
=\widehat{T_{2}}'
.\]
\noindent Secondly, one inserts $2$ in the tableau $\widehat{T_{2}}'$:

\noindent we have $1\mathbf{3\overline{3}2}\equiv \mathbf{132}\overline{3}\equiv 312\overline{3},$ then 
\[\mathtt{2}\rightarrow
\begin{tabular}
[c]{|l|}\hline
$\mathtt{1}$\\\hline
$\mathtt{3}$\\\hline
$\mathtt{\overline{3}}$\\\hline
\end{tabular}
=
\begin{tabular}
[c]{|l|l}\hline
$\mathtt{1}$ & \multicolumn{1}{|l|}{$\mathtt{3}$}\\\hline
$\mathtt{2}$\\\cline{1-1}
$\mathtt{\overline{3}}$\\\cline{1-1}
\end{tabular}
.\]
We have $133\equiv 313,$ then 
\[\mathtt{3}\rightarrow
\begin{tabular}
[c]{|l|}\hline
$\mathtt{1}$\\\hline
$\mathtt{3}$\\\hline
\end{tabular}
=
\begin{tabular}
[c]{|l|l}\hline
$\mathtt{1}$ & \multicolumn{1}{|l|}{$\mathtt{3}$}\\\hline
$\mathtt{3}$\\\cline{1-1}
\end{tabular}
.\]
We have $23\equiv 23,$ then 
\[\mathtt{3}\rightarrow
\begin{tabular}
[c]{|l|}\hline
$\mathtt{2}$\\\hline
\end{tabular}
=
\begin{tabular}
[c]{|l|}\hline
$\mathtt{2}$\\\hline
$\mathtt{3}$\\\hline
\end{tabular}
.\]
Hence,
\[
\mathtt{2}\rightarrow \widehat{T_{2}}'=
\begin{tabular}
[c]{|l|lll}\hline
$\mathtt{1}$ & \multicolumn{1}{|l|}{$\mathtt{1}$} & \multicolumn{1}{|l|}{$\mathtt{2}$} & \multicolumn{1}{|l|}{$\mathtt{3}$}\\\hline
$\mathtt{2}$ & \multicolumn{1}{|l|}{$\mathtt{3}$} & \multicolumn{1}{|l|}{$\mathtt{3}$}\\\cline{1-3}
$\mathtt{\overline{3}}$\\\cline{1-1}
\end{tabular}
=\mathtt{\overline{3}}\rightarrow T_{2}
.\]
\end{example}

\begin{remark}
Consider a word $w$ in $\mathcal{C}_{n}^{\ast}.$ The symplectic tableau $P(w)$ can be computed by starting with the empty word, which is a valid tableau, and iteratively applying the insertion schemes described above. Notice that when $w$ is the reading of a symplectic tableau $T$, we have $P(w) = T$.
\medskip

Let $u$ and $v$ be the  readings of two admissible columns $U$ and $V$ respectively. As we have seen in Subsection~\ref{symplectic tableaux}, $U\succeq V$ means that the column $U$ can appear to the right of $V$ in a symplectic tableau. Note that   $U\nsucceq V$ means that the word $uv$ is not the reading of a symplectic tableau.
\end{remark}

\begin{lemma}
\label{First Lemma}
Let $u$ and $v$ be the  readings of two admissible columns  $U$ and $V$ respectively. The symplectic tableau $P(uv)$ consists of at most two columns.
\end{lemma}

\begin{proof}
For $U \succeq V$, the result is trivial. Let $u = x_{1}\ldots x_{p}$ and $v  = y_{1}\ldots y_{q}$ be respectively the  readings of  two  admissible columns  $U$ and $V$ of height $p$ and $q$, such that $U\nsucceq V$. Let $u^{0}z_{1}\ldots z_{q}$ be the highest weight vertex of the connected component containing $uv$. We begin inserting the first element $y_{1}$ of $v$ in the column $U$. The shape of $P(uy_{1})$ depends of the connected component  containing $uy_{1}$. The highest weight vertex of this component is $u^{0}z_{1}$. By Lemma~\ref{highestweight}, $u^{0}$ is of highest weight and $\varepsilon_{i}(z_{1})\leq \varphi_{i}(u^{0})$, for any $i=1,\ldots ,n$. Then we obtain the following cases.
\medskip

\noindent {\bf Case~1:}  $u^{0}z_{1} = 1\ldots p(p+1)$.  In this case, $uy_{1}$ is an admissible column word, $z_{1}=p+1$ and $\mathrm{wt}(z_{1})$ =$ \Lambda_{p+1} - \Lambda_{p}$. Then during the insertion of the letter $y_{1}$ in the column $U$, this column of height $p$ corresponding to the weight $\Lambda_{p}$ is transformed into a column  of height $p+1$ corresponding to the weight $\Lambda_{p+1}$. Its reading is $uy_{1}$.  After one continues inserting the others elements $y_{2},\ldots, y_{q}$ of the column word $v$. We know by the definition of an admissible column that every element of this column is strictly larger than its preceding, then we have two cases:

First, suppose that $z_{i} = p+i$, for $i=2,\ldots, q$. Then $\mathrm{wt}(z_{i}) = \Lambda_{p+i} - \Lambda_{p+i-1}$ and during the insertion of $y_{i}$ in the column of reading  $uy_{1}\ldots y_{i-1}$, this column of height $p+i-1$ is turned into the column of reading $uy_{1}\ldots y_{i}$  and of height $p+i$. Thus $uv$ is an admissible column word and $P(uv)$ consists of one column $uv$. 

Second, suppose that there exists an element $y_{k}$ of the column word $v$ such that $uy_{1}\ldots y_{k-1}y_{k}$ is a nonadmissible column word whose each strict factor is an admissible column word, then $z_{k} = \overline{p+k-1}$ and $\mathrm{wt}(z_{k}) = \Lambda_{p+k-2}-\Lambda_{p+k-1}$, then during the insertion of $y_{k}$ in the admissible column of reading $uy_{1}\ldots y_{k-1}$, this column of height $p+k-1$ is transformed into a column of height $p+k-2$. After one continues inserting the remaining elements of $v$, then one adds those letters  in distinct rows in  the considered column or one removes some letters  from distinct rows of the same column.

\noindent Hence, in this case $P(uv)$ consists of one column.
\medskip

\noindent {\bf Case~2:}  $u^{0}z_{1} = 1\ldots p\overline{p}$.  In this case, $uy_{1}$ is a nonadmissible column word such that each strict factor is an admissible column word. We have $\mathrm{wt}(z_{1}) = \Lambda_{p-1}-\Lambda_{p}$, then during the insertion  of $y_{1}$ in the admissible column $U$, this column of height $p$ is turned into a column of height $p-1$. Since the elements of  the column $V$ are strictly increasing, one can prove by similar arguments of  Case~1, that during the computation of $P((uy_{1})y_{2}\ldots y_{q})$, one adds a number of boxes of the considered column in distinct rows  and  one removes some boxes from distinct rows of the same column. Note also that the column $U$ can be contracted to become empty.
Hence, we have in this case that $P(uv)$ consists of one column or zero columns.
\medskip

\noindent{\bf Case~3:}  $u^{0}z_{1} = 1\ldots p1$.  In this case, $uy_{1}$ is not a column word, then during the insertion of $y_{1}$ in the admissible column $U$, an element appears in a second column. After,  one inserts the next element $y_{2}$ of the column $V$ in $P(uy_{1})$, the highest weight of the connected component containing $\mathrm{w}(P(uy_{1}))y_{2}$ may be written $\mathrm{w}(P(uy_{1})^{0})z_{2}$, where  $\mathrm{w}(P(uy_{1})^{0})$ is of highest weight and by Lemma~\ref{highestweight}, we have:

\noindent {(i)} $z_{2} = i$ (with $i = p+1$ or $i=2$), then its weight is equal to $\Lambda_{i} - \Lambda_{i-1}$, then during the insertion $y_{2}\rightarrow P(uy_{1})$ a column of height $i-1$ is turned into a column of height $i$. Then one adds a box in the left column or in the right column of $P(uy_{1})$.
\medskip

\noindent{(ii)} $z_{2} = \overline{p}$, then its weight is equal to $\Lambda_{p-1} - \Lambda_{p}$, then during the insertion $y_{2}\rightarrow P(uy_{1})$, the right column of height $p$ is turned into a column of height $p-1$.
\medskip

After we continue inserting the remaining letters of $v$, and since every element is strictly larger than its preceding, one adds boxes in distinct rows in the right or in the left column and similarly one removes boxes from distinct rows of the considered symplectic tableau. Note also that it is impossible that one of the columns contracts to become empty. Indeed, let $u$ and $v$ be respectively the readings of  two admissible columns $U$ and $V$ such that $uv$ is of highest weight. Suppose that after adding $k$ boxes in the right column, one inserts $p$ boxes in the left column to contract it into an empty one. Then in this case we have $u=1\ldots p$ and $v=1\ldots k\overline{p}(\overline{p-1})\ldots\overline{1}$. We have in the word $v$ that    $N(p)=p+k>p$. So the column $V$ is not admissible, which yields a contradiction.

Hence, $P(uv)$ consists of two columns.
\end{proof}

\begin{lemma}
\label{Second Lemma}
Let $u$ and $v$ be the readings of two admissible columns $U$ and $V$ respectively, such that $U \not\succeq V$. Suppose that $P(uv)$ has two columns and let $W$ be the rightmost column. Then the column $U$ contains more elements than $W$.
\end{lemma}

\begin{proof}
Let $u = x_{1}\ldots x_{p}$ and $v  = y_{1}\ldots y_{q}$ be respectively the  readings of two  admissible columns  $U$ and $V$ of height $p$ and $q$,  such that $U\nsucceq V$. Let $w$ and $w'$ be respectively the  readings of the right and left column $W$ and $W'$ of $P(uv)$. If the height of $U$ is greater than the height of $V$, then in all cases we have $h(W)<p$.

Suppose that $q\geq p$ and the columns $U$ and $V$ contain only unbarred letters. Suppose that during the computation of $P(uv)$, we only add boxes by applying relations of type $(R_{1})$. In other words, we compute $P(uv)$ by Schensted's insertion. If $h(W)=p$, then during inserting the first $p$ elements of $V$, $p$ boxes are added in the second column and they are all filled by elements of $U$. Since the number of added boxes is equal to the height of $U$, $\textrm{w}(P(uv)) = uv$. Then $U\succeq V$ which yields a contradiction. Hence, $h(W)<p$.

Suppose now that during the computation of $P(uv)$, we only add boxes by applying relations of type $(R_{1})$ or $(R_{2})$. By definition of $P(uv)$  we have $\textrm{w}(P(uv))= ww'\equiv uv$. Then the words $uv$ and $ww'$ occur at the same place in their isomorphic connected components $B(uv)$ and $B(ww')$ of the crystal  $G_{n}$. Note that all the vertices in a connected component are the readings of tableaux of  the  same shape. Let $(uv)^{0}$ and $(ww')^{0}$ be respectively the highest weight vertices of $B(uv)$ and $B(ww')$. By Remark~\ref{Highestweightremark}, the word $(ww')^{0}$ is the reading of a tableau that all its elements are unbarred letters, then  $(uv)^{0}$ and $(ww')^{0}$ are related by relations of type $(R_{1})$. Hence, as we have seen above,  the height of the  second column of $P((uv)^{0})$ is strictly less than $p$. Since $(ww')^{0}$ and $ww'$  are the readings of  two symplectic tableaux of the same shape, the length of $w$ is stricly less than $p$.

Suppose that during the insertion of the first $k$ elements of $v$, for $k\leq p-1$,  into the column  $U$, we add $k$ boxes in a second column. Then 
\[P(uy_{1}\ldots y_{k})= 
\begin{tabular}
[c]{|l|l}\hline
$\mathtt{U_{1}}$ &   \multicolumn{1}{|l|}{$\mathtt{U_{2}}$} \\\hline
\end{tabular},\]
where $U_{1}$ contains $p$ elements and $U_{2}$ contains the $k$ added boxes. After we insert $y_{k+1}$ in the column $U_{1}$. Suppose that $\textrm{w}(U_{1})y_{k+1}$ is a nonadmissible column  word such that all of its proper factors are admissible. Let $\widetilde{\textrm{w}(U_{1})y_{k+1}}$  be the column word obtained  from $\textrm{w}(U_{1})y_{k+1}$ after applying relation of type $(R_{3})$. Then we insert the elements of $\widetilde{\textrm{w}(U_{1})y_{k+1}}$ in the column $U_{2}$. This insertion does not cause a new contraction. Then if we obtained two columns, the height  of the right one is strictly less than the height of $U_{2}$ which is strictly less than $p$. After we continue inserting the remaining elements of $v$, and the height of the right column of the final tableau  is strictly less than $p$.  
\end{proof}

\section{Convergent presentation of plactic monoid for type C}
\label{Convergent presentation}

\subsection{Knuth-like presentation}

Consider a presentation of the plactic monoid  $\mathbf{P_{n}}(C)$, by the $2$-polygraph $\Sigma^{\text{Sp(n)}}$, whose  set of $1$-cells is $\mathcal{C}_{n}$ and  whose $2$-cells correspond to the relations $(R_{1})$, $(R_{2})$ and $(R_{3})$ oriented with respect to the reverse deglex order, that is
\[
\begin{array}{rl}
\Sigma^{\text{Sp(n)}}_2 &\:=\:\big\{\;  xzy\odfl{\kappa_{x,y,z}} zxy \;\big|\; x< y \leq z \text{ } \text{ and } \text{ } z\neq \overline{x} \;\big\} \\
                          &\;\cup\;\big\{\; yxz \odfl{\kappa_{x,y,z}'}yzx  \;\big|\; x\leq y< z \text{ }  \text{ and } \text{ }   z\neq \overline{x}  \;\big\}\\
        &\;\cup\;\big\{\;  yx\overline{x}\odfl{\xi_{x,y,\overline{x}}}y\overline{(x-1)}(x-1)  \;\big|\;  x\leq y\leq\overline{x}\text{ }\text{ and }\text{ } 1<x\leq n\;\big\}\\
        &\;\cup\;\big\{\;  x\overline{x}y \odfl{\xi_{x,y,\overline{x}}'}\overline{(x-1)}(x-1)y \;\big|\;  x\leq y\leq\overline{x}\text{ }\text{ and }\text{ } 1<x\leq n\;\big\}\\
        &\;\cup\;\big\{\; w \odfl{\zeta_{w}}\widetilde{w} \;\big|\; \text{w and } \widetilde{w} \text{ satisfy the conditions of the relation } (R_{3})  \;\big\}.
\end{array}
\]
The order being monomial, the $2$-polygraph $\Sigma^{\text{Sp(n)}}$ is terminating.

\begin{remark}
For $n\geq 4$, the Knuth presentation of the plactic monoid for type A doesn't admit a finite completion compatible with the reverse deglex order. Indeed,  by similar arguments used in~\cite{KubatOkninski14},  one can show that during the completion one adds an infinity of $2$-cells of the form $232^{i}124\odfl{} 2342^{i}12$, for $i>1$. The $2$-polygraph $\Sigma^{\text{Sp(n)}}$  contains the Knuth relations for type A and  we can not apply relations  of type $(R_{2})$ and $(R_{3})$ on the words $232^{i}124$ and $2342^{i}12$, for $i>1$, then the $2$-polygraph $\Sigma^{\text{Sp(n)}}$ does not also admit a finite completion compatible with the reverse deglex order.
\end{remark}

\subsection{Column presentation}

In order to give a finite convergent presentation of the plactic monoid $\mathbf{P_{n}}(C)$, one introduces the admissible column generators. The set of generators is 
\[
\Gamma_{1}=\big\{\; c_{u} \; \big| \; \text{ $u$ is a nonempty admissible column word  of } \mathcal{C}_{n}^{*}  \; \big\},
\]
where each symbol $c_{u}$ represents the element $u$ of $\mathbf{P_{n}}(C)$.
In particular,  the word $c_{x}$ represents the letter $x$ in $\mathcal{C}_{n}$, hence the set $\Gamma_{1}$ also generates $\mathbf{P_{n}}(C)$. 

Let $w=x_{1}\ldots x_{l(w)}$ and $\widetilde{w}= \widetilde{x}_{1}\ldots \widetilde{x}_{l(\widetilde{w})}$ be  two columns such that $w\equiv \widetilde{w}$ by a relation of type $(R_{3})$. 

We consider the two following sets of $2$-cells, the $2$-cells corresponding to the  relations $(R_{1}), (R_{2})$ and $(R_{3})$, that is,  
\[
\begin{array}{rl}
\Gamma^{\text{Sp(n)}}_2 &\:=\:\big\{\; c_xc_zc_y\odfl{c_{\kappa_{x,y,z}}} c_zc_xc_y  \;\big|\; x< y \leq z \text{ } \text{ and } \text{ } z\neq \overline{x} \;\big\} \\
                          &\;\cup\;\big\{\; c_yc_xc_z\odfl{c_{\kappa_{x,y,z}'}} c_yc_zc_x  \;\big|\; x\leq y< z \text{ }  \text{ and } \text{ }   z\neq \overline{x}  \;\big\}\\
        &\;\cup\;\big\{\;  c_yc_xc_{\overline{x}}\odfl{c_{\xi_{x,y,\overline{x}}}}c_yc_{\overline{(x-1)}}c_{(x-1)}  \;\big|\;  x\leq y\leq\overline{x}\text{ }\text{ and }\text{ } 1<x\leq n\;\big\}\\
        &\;\cup\;\big\{\;   c_{x}c_{\overline{x}}c_{y}\odfl{c_{\xi_{x,y,\overline{x}}'}} c_{\overline{(x-1)}}c_{(x-1)}c_{y}\;\big|\;  x\leq y\leq\overline{x}\text{ }\text{ and }\text{ } 1<x\leq n\;\big\}\\
        &\;\cup\;\big\{\; c_{x_{1}}\ldots c_{x_{l(w)}} \odfl{c_{\zeta_{w}}}c_{\widetilde{x}_{1}}\ldots  c_{\widetilde{x}_{l(\widetilde{w})}} \;\big|\;    w \text{ and } \widetilde{w}  \text{ verify the relation}~(R_{3})  \;\big\},
\end{array}
\]
and the $2$-cells corresponding to the defining relations for the extra column generators $c_{u}$, where $u$ is a nonempty admissible column word  of $\mathcal{C}_{n}^{*}$  with $l(u)~\geq~2$,
\[
\Gamma^{\text{c}(n)}_2
\: = \:
\big\{ \;c_{y_{1}}\ldots c_{y_{k}}\odfl{\gamma_{y_1,\ldots,y_k}} c_{y_{1}\ldots y_{k}} \; \big| \; y_{1}\ldots y_{k} \text{ is a nonempty admissible column} \; \big\}.
\]
The monoid $\mathbf{P_{n}}(C)$ is  presented by the $2$-polygraph $\Gamma^{(n)} = (\Gamma_{1}, \Gamma_2^{(n)})$, with $\Gamma_2^{(n)}~=~\Gamma^{\text{Sp}(n)}_2~\cup~\Gamma^{\text{c}(n)}_2$.
 \vspace*{1cm}

Let $u$ and $v$ be respectively the readings of two nonempty admissible columns $U$ and $V$. Suppose that $U\nsucceq V$, by Lemma~\ref{First Lemma} the symplectic tableau $P(uv)$ consists of at most two columns. Define a $2$-cell
\begin{itemize}
\item $c_{u}c_{v}\odfl{\alpha_{u,v}} c_{w}c_{w'}$, where  the words $w$ and $w'$ are respectively the  readings of the right and left columns $W$ and $W'$ of $P(uv)$ if this symplectic tableau consists of two columns.
\item $c_{u}c_{v}\odfl{\alpha_{u,v}} c_{w}$, where $w$ is the reading of the column $W$ of $P(uv)$ if it consists of one column.
\item $c_{u}c_{v}\odfl{\alpha_{u,v}}c_{\epsilon}$, where $\epsilon$ is the empty word if $P(uv)$ consists of zero columns.
\end{itemize}

Define
\[
\Omega_2
\: = \:
\big\{ \;c_{u}c_{v}\odfl{\alpha_{u,v}} c_{w}c_{w'} \; \big| \; u \text{ and }v \text{ are nonempty admissible columns words of } \mathcal{C}_{n}^{*} \text{ such that } U\nsucceq V \big\}.
\]
The $2$-polygraph $\Sigma^{\text{acol(n)}} = \left(\Gamma_{1}, \Omega_{2}\right)$ is called the \emph{column presentation}.

\begin{remark}
Every rule in $\Omega_{2}$ holds in the symplectic plactic monoid $\mathbf{P_{n}}(C)$, indeed,
$$c_{u}c_{v}\equiv uv\equiv \textrm{w}(P(uv))= ww'\equiv c_{w}c_{w'}.$$ 
\end{remark}

\noindent Let $<$ be the total order on $\mathcal{C}_{n}$ defined by $1<2<\ldots <n<\overline{n}<\ldots<\overline{1}$. Denote by $<_{\text{deg}}$ the deglex order induced by $<$ on the monoid $\mathcal{C}_{n}^{\ast}$. Let us define an order on $\Gamma_{1}^{*}$.  First, let $\sqsubset$ be the total order  on $\Gamma_{1}$ defined  by $$c_{u}~\sqsubset~c_{v} \text{ if }\quad  l(u)~<~l(v)\quad \text{or}\quad [l(u)=l(v)\text{ }\text{and}\text{ } {u<_{\text{lex}} v}]. $$ 
Secondly, consider the order $\prec$ on $\Gamma_{1}^{*}$, defined as follows. We have 
\[\begin{array}{rl}
&c_{u_{1}}c_{u_{2}}\ldots c_{u_{m}} \prec c_{v_{1}}c_{v_{2}}\ldots c_{v_{n}} \text{ if } \text{  } m<n  \\
& \text{or} \text{ } \left( \text{m} = \text{n} \text{ } \text{and there exists} \text{ }  i \text{ }\text{such that }\text{ } c_{u_{i}}\sqsubset c_{v_{i}} \text{ } \text{ and } \text{ } \forall j<i, \text{ } c_{u_j} = c_{v_j}\right),
\end{array}\]
where $c_{u_{i}}$ and $c_{v_{j}}$ are elements of $\Gamma_{1}$, for $i=1,\ldots, m$ and $j=1,\ldots, n$. That is, two elements of $\Sigma_{1}^{\ast}$ are compared using  the number of theirs symbols. If they have the same number of symbols, we compare them using the total order $\sqsubset$ on the elements of $\Gamma_{1}$ which is induced by the deglex order on the columns words of $\mathcal{C}_{n}^{\ast}$. Then $\prec$ is a total order on $\Gamma_{1}^{*}$ and it is a well-ordering.

\begin{lemma}
\label{Finited lemma}
The $2$-polygraph $\Sigma^{\text{acol(n)}}$ is finite.
\end{lemma}

\begin{proof}
The set $\Gamma_{1}$ is finite thanks to the fact that the admissible columns words of $\mathcal{C}_{n}^{*}$ have  length at most $n$. Hence, the $2$-polygraph $\Sigma^{\text{acol(n)}}$ is finite.
\end{proof}

The following lemma shows that the plactic monoid $\mathbf{P_{n}}(C)$ is presented by the $2$-polygraph $\Sigma^{\text{acol(n)}}$ :

\begin{lemma}
The $2$-polygraphs $\Gamma^{(n)}$ and $\Sigma^{\text{acol(n)}}$ are Tietze equivalent.
\end{lemma}

\begin{proof}
Every relation in $\Gamma_{2}^{\text{Sp(n)}}$ can be deduced from rules in $\Omega_{2}$, indeed, the $2$-cells $c_{\kappa_{x,y,z}}$ for $x<y\leq z$ and $z\neq \overline{x}$, $c_{\kappa'_{x,y,z}}$ for $x\leq y<z$ and $z\neq \overline{x}$,  $c_{\xi_{x,y,\overline{x}}}$ and $c_{\xi'_{x,y,\overline{x}}}$ for $x\leq y\leq \overline{x}$ and $1<x\leq n$ are obtained from rules in  $\Omega_{2}$ according to the following diagrams
\[
\xymatrix @C=6em @R=4em{
{c_{z}c_{x}c_{y}}
   \ar@2 [d]_-{c_{z}\alpha_{x,y}}
& {c_{x}c_{z}c_{y}}
     \ar@2 [l] _-{c_{\kappa_{x,y,z}}}
     \ar@2 [d] ^-{\alpha_{x,z}c_{y}}
\\
{c_{z}c_{xy}}
& {c_{xz}c_{y}}
    \ar@2[l] ^-{\alpha_{xz,y}}
}
\qquad
\xymatrix @C=6em @R=4em{
{c_{y}c_{z}c_{x}}
     \ar@2 [d]_-{\alpha_{y,z}c_{x}}
& {c_{y}c_{x}c_{z}}
      \ar@2 [l] _-{c_{\kappa'_{x,y,z}}}
     \ar@2 [d] ^-{c_{y}\alpha_{x,z}}
\\
{c_{yz}c_{x}}
\ar@2[r] _-{\alpha_{yz,x}}
& {c_{y}c_{xz}}
    }    
\]
\[
\xymatrix @C=4em @R=4em{
{c_{y}c_{\overline{x-1}}c_{x-1}}
     \ar@2 [d]_-{\alpha_{y,\overline{x-1}}c_{x-1}}
& {c_{y}c_{x}c_{\overline{x}}}
      \ar@2 [l] _-{c_{\xi_{x,y,\overline{x}}}}
     \ar@2 [d] ^-{c_{y}\alpha_{x,\overline{x}}}
\\
{c_{y(\overline{x-1})}c_{x-1}}
\ar@2[r] _-{\alpha_{y(\overline{x-1}),x-1}}
& {c_{y}c_{x\overline{x}}}
}
\qquad
\xymatrix @C=3em @R=4em{
{c_{\overline{x-1}}c_{x-1}c_{y}}
     \ar@2 [d]_-{c_{\overline{x-1}}\alpha_{x-1,y}}
& {c_{x}c_{\overline{x}}c_{y}}
     \ar@2 [l] _-{c_{\xi'_{x,y,\overline{x}}}}
     \ar@2 [d] ^-{\alpha_{x,\overline{x}}c_{y}}
\\
{c_{\overline{x-1}}c_{(x-1)y}}
& {c_{x\overline{x}}c_{y}}
\ar@2[l] _-{\alpha_{x\overline{x},y}}
    }    
\]
Let $w=x_{1}\ldots x_{p}\ldots x_{q}\ldots x_{k}$ be a nonadmissible column word of length $k$ such that each strict factor of it is an admissible column word. Let $z=x_{p}$ be the lowest unbarred letter such that  $z=x_{p}$ and  $\overline{z}=x_{q}$ occur in $w$ and $N(z) = z+1$. Then the $2$-cell $c_{\zeta_{w}}$ is deduced from from rules in $\Omega_{2}$ according to the following diagram
\[
\xymatrix @C=2em @!C{
c_{x_{1}}\ldots c_{x_{p}}\ldots c_{x_{q}}\ldots c_{x_{k}} 
	\ar@2[rr] ^{c_{\zeta_{w}}}
	\ar@2[dd]|{\alpha_{x_{1},x_{2}}c_{x_{3}}\ldots c_{x_{k}}}
&& c_{x_{1}}\ldots \widehat{c_{x_{p}}} \ldots \widehat{c_{x_{q}}}\ldots c_{x_{k}} 
	\ar@2[dd]|{\alpha_{x_{1},x_{2}}c_{x_{3}}\ldots c_{x_{k}}}
\\
\\
(\ldots)
\ar@2[dd]|{\alpha_{x_{1}\ldots x_{k-2},x_{k-1}}c_{x_{k}}}
&& (\ldots )
\ar@2[dd]|{\alpha_{x_{1}\ldots \hat{x_{p}}\ldots \hat{x_{q}}\ldots x_{k-1},x_{k}}}
\\
\\
{c_{x_{1}\ldots x_{p}\ldots x_{q}\ldots x_{k-1}}c_{x_{k}}}
	\ar@2[rr]^-{ \alpha_{x_{1}\ldots x_{p}\ldots x_{q}\ldots x_{k-1}, x_{k}} }
&& {c_{x_{1}\ldots \widehat{x_{p}}\ldots \widehat{x_{q}}\ldots x_{k}}}
}
\]
where the  symbol  $\hat{x}$ means that $x$ is removed.

In addition, any rules $\gamma_{y_1,\ldots,y_k}$ in $\Gamma^{\text{c}(n)}_2$ can be obtained using those in $\Omega_{2}$, according to the following diagram
\[
\xymatrix @C=0.75em @!C{
c_{y_{1}}\ldots c_{y_{k}} 
	\ar@2[rrrr] ^{\gamma_{y_1,\ldots,y_k}}
	\ar@2[dd]|{\alpha_{y_{1},y_{2}}c_{y_{3}}\ldots c_{y_{k}}}
&&&& c_{y_{1}\ldots y_{k}}
\\
\\
{c_{y_{1}y_{2}}}c_{y_{3}}\ldots c_{y_{k}}
	\ar@2[rr]^{\alpha_{y_{1}y_{2},y_{3}}c_{y_{4}}\ldots c_{y_{k}}}
&&(\cdots)
	\ar@2[rr]^{\alpha_{y_{1}\ldots y_{k-2},y_{k-1}}c_{y_{k}}}
&& c_{y_{1}\ldots y_{k-1}}c_{y_{k}}
	\ar@2[uu] |{\alpha_{y_{1}\ldots y_{k-1},y_{k}}}
}
\]
\newline

\end{proof}

\begin{theorem}
\label{maintheorem}
The $2$-polygraph $\Sigma^{\text{acol(n)}}$ is a finite convergent presentation of the monoid $\mathbf{P_{n}}(C)$.
\end{theorem}

\begin{proof}
By Lemma~\ref{Finited lemma}, the $2$-polygraph $\Sigma^{\text{acol(n)}}$ is finite. Let us show that it is also convergent.
First, in order to prove the termination of $\Sigma^{\text{acol(n)}}$, we show that if $h\odfl{} h'$  then $h'\prec h$. One finds two cases.

\noindent {First case :} let $h=pc_{u}c_{v}q$ and $h'=pc_{w}q$, with $p,q\in \Gamma_{1}^{*}$ and $c_{u}, c_{v}, c_{w} \in\Gamma_{1}$. One remarks that $h'$ is shorter than $h$, then $h'\prec h$.

\noindent{Second case :} let $h=pc_{u}c_{v}q$ and $h'=pc_{w}c_{w'}q$, with $p,q\in \Gamma_{1}^{*}$ and $c_{u},c_{v},c_{w},c_{w'}\in\Gamma_{1}$, where $w$ and $w'$ are respectively  the readings of the right and left columns of $P(uv)$.  One remarks that $h$ and $h'$ have the same length. By Lemma~\ref{Second Lemma} the length of $u$ is strictly larger than the length of $w$, then $c_{w}\sqsubset c_{u}$. Consider $i=l(p)+1$, $c_{u_{i}}=c_{w}$ and $c_{u}=c_{v_{i}}$ then we have $c_{u_{i}}\sqsubset c_{v_{i}}$ and for all $j<i$, $c_{u_{j}} = c_{v_{j}}$. Hence $h'\prec h$. Since every application of a $2$-cell of $\Omega_{2}$ yields a $\prec$-preceding word, it follows that any sequence of rewriting using $\Omega_{2}$ must terminate.
Hence, the $2$-polygraph  $\Sigma^{\text{acol(n)}}$ is terminating.
\medskip

\noindent Let us show the confluence of the $2$-polygraph $\Sigma^{\text{acol(n)}}$. Let $h\in \Gamma_{1}^{*}$ and  $h'$, $h''$ be two normal forms obtained from $h$. We have to prove that $h'=h''$. Suppose that $h'=c_{u_{k}}\ldots c_{u_{1}}$. Since $h'$ is a normal form, the words $u_{1}, \ldots, u_{k}$ are respectively the readings of $k$  admissible columns $U_{1},\ldots, U_{k} $ of a symplectic tableau, i.e, $U_{i}\preceq U_{i+1}, \forall i$. Then $u_{k}\ldots u_{1}=\textrm{w}(T')$, where $T'$ is the unique symplectic tableau  such that $$ \textrm{w}(T')= u_{k}\ldots u_{1}\equiv h'.$$
Similarly, $h'' = c_{v_{l}}\ldots c_{v_{1}}$ is a normal form, then there exists a unique symplectic tableau $T''$ such that  
$$ \textrm{w}(T'')= v_{l}\ldots v_{1}\equiv h''.$$
Since $h\equiv h'\equiv h''$, we have  by Theorem~\ref{Cross section} that $T'=T''$. Then  we have $k=l$ and $u_{i}=v_{i}, \forall i=1,\ldots, k.$ Thus $h'=h''.$ 

Hence, the $2$-polygraph $\Sigma^{\text{acol(n)}}$ is convergent.
\end{proof}

\subsection{Finiteness properties of plactic monoid of type C}

A monoid is of \emph{finite derivation type} ($\mathrm{FDT}_3$) if it admits a finite presentation whose relations among the relations are finitely generated, see~\cite{Squier94}. The property $\mathrm{FDT}_3$ is a natural extension of the properties of being finitely generated $(\mathrm{FDT}_1)$ and finitely presented $(\mathrm{FDT}_2)$. Using the notion of polygraphic resolution,  one can define the higher-dimensional finite derivation type properties $\mathrm{FDT}_{\infty}$, see~\cite{GuiraudMalbos12advances}. They generalise in any dimension the finite derivation type $\mathrm{FDT}_{3}$.  A monoid is said to be $\mathrm{FDT}_{\infty}$ if it admits a finite polygraphic resolution. By Corollary~4.5.4 in~\cite{GuiraudMalbos12advances}, a monoid with a finite convergent presentation is $\mathrm{FDT}_{\infty}$. Then by Theorem~\ref{maintheorem}, we have

\begin{proposition}
Plactic monoids of type C  satisfy the homotopical finiteness condition $\mathrm{FDT}_{\infty}$.
\end{proposition}

In the homological way, a monoid $M$ is of \emph{homological type} $\mathrm{FP}_{\infty}$ when there exists a resolution of $M$ by projective, finitely generated $\mathbb{Z}M$-modules. By Corollary~5.4.4 in~\cite{GuiraudMalbos12advances} the property $\mathrm{FDT}_{\infty}$ implies the property $\mathrm{FP}_{\infty}$. Hence we have

\begin{proposition}
Plactic monoids of type C satisfy the homological finiteness property type $\mathrm{FP}_{\infty}$.
\end{proposition}

Starting with the column presentation $\Sigma^{\text{acol(n)}}$ of the monoid $\mathbf{P_{n}}(C)$, we hope  to construct  a polygraphic resolution of  $\mathbf{P_{n}}(C)$ by studying the confluence of all the critical branchings of  the presentation.

\section*{Acknowledgments}

The author wishes to thank the anonymous referees for their comments and   meaningful suggestions that helped to produce an improved version of this  work. Special thanks are due to Stephane Gaussent and Philippe Malbos for their advices and comments as well as for every useful discussions.


\begin{small}
\renewcommand{\refname}{\Large\textsc{References}}
\bibliographystyle{alpha}
\bibliographystyle{plain}
\bibliography{biblioarticletypeC}

\def\cprime{$'$}
\begin{thebibliography}{JMMO91}

\bibitem[Ani86]{Anick86}
David~J. Anick.
\newblock On the homology of associative algebras.
\newblock {\em Trans. Amer. Math. Soc.}, 296(2):641--659, 1986.

\bibitem[Bak00]{Baker00}
T.~H. Baker.
\newblock An insertion scheme for {$C_n$} crystals.
\newblock In {\em Physical combinatorics ({K}yoto, 1999)}, volume 191 of {\em
  Progr. Math.}, pages 1--48. Birkh\"auser Boston, Boston, MA, 2000.

\bibitem[BCCL15]{BokutChenChenLi15}
L.~A. Bokut, Yuqun Chen, Weiping Chen, and Jing Li.
\newblock New approaches to plactic monoid via {G}r\"obner--{S}hirshov bases.
\newblock {\em J. Algebra}, 423:301--317, 2015.

\bibitem[Bou68]{BourbakiLie4-6}
N.~Bourbaki.
\newblock {\em \'{E}l\'ements de math\'ematique. {F}asc. {XXXIV}. {G}roupes et
  alg\`ebres de {L}ie. {C}hapitre {IV}: {G}roupes de {C}oxeter et syst\`emes de
  {T}its. {C}hapitre {V}: {G}roupes engendr\'es par des r\'eflexions.
  {C}hapitre {VI}: syst\`emes de racines}.
\newblock Actualit\'es Scientifiques et Industrielles, No. 1337. Hermann,
  Paris, 1968.

\bibitem[CGM14]{CainGrayMalheiro14}
A.~J. {Cain}, R.~D. {Gray}, and A.~{Malheiro}.
\newblock {Crystal bases, finite complete rewriting systems, and biautomatic
  structures for Plactic monoids of types \$A\_n\$, \$B\_n\$, \$C\_n\$,
  \$D\_n\$, and \$G\_2\$}.
\newblock {\em ArXiv e-prints}, December 2014.

\bibitem[CGM15]{CainGrayMalheiro15}
Alan~J. Cain, Robert~D. Gray, and Ant{\'o}nio Malheiro.
\newblock Finite {G}r\"obner--{S}hirshov bases for {P}lactic algebras and
  biautomatic structures for {P}lactic monoids.
\newblock {\em J. Algebra}, 423:37--53, 2015.

\bibitem[Ful97]{Fulton97}
William Fulton.
\newblock {\em Young tableaux}, volume~35 of {\em London Mathematical Society
  Student Texts}.
\newblock Cambridge University Press, Cambridge, 1997.
\newblock With applications to representation theory and geometry.

\bibitem[GGM15]{GaussentGuiraudMalbos14}
St{\'e}phane Gaussent, Yves Guiraud, and Philippe Malbos.
\newblock Coherent presentations of {A}rtin monoids.
\newblock {\em Compos. Math.}, 151(5):957--998, 2015.

\bibitem[GM12]{GuiraudMalbos12advances}
Yves Guiraud and Philippe Malbos.
\newblock Higher-dimensional normalisation strategies for acyclicity.
\newblock {\em Adv. Math.}, 231(3-4):2294--2351, 2012.

\bibitem[GM14]{GuiraudMalbos14}
Yves Guiraud and Philippe Malbos.
\newblock Polygraphs of finite derivation type.
\newblock submitted, arXiv:1402.2587, 2014.

\bibitem[JMMO91]{JimboMisraMiwaOkado91}
Michio Jimbo, Kailash~C. Misra, Tetsuji Miwa, and Masato Okado.
\newblock Combinatorics of representations of $\textsc{U}_q(\widehat{sl}(n))$
  at $q=0$.
\newblock {\em Comm. Math. Phys.}, 136(3):543--566, 1991.

\bibitem[Kas91]{Kashiwara91}
Masaki Kashiwara.
\newblock Crystallizing the {$q$}-analogue of universal enveloping algebras.
\newblock In {\em Proceedings of the {I}nternational {C}ongress of
  {M}athematicians, {V}ol.\ {I}, {II} ({K}yoto, 1990)}, pages 791--797. Math.
  Soc. Japan, Tokyo, 1991.

\bibitem[Kas95]{Kashiwara94}
Masaki Kashiwara.
\newblock On crystal bases.
\newblock In {\em Representations of groups ({B}anff, {AB}, 1994)}, volume~16
  of {\em CMS Conf. Proc.}, pages 155--197. Amer. Math. Soc., Providence, RI,
  1995.

\bibitem[KN85]{KapurNarendran85}
Deepak Kapur and Paliath Narendran.
\newblock A finite {T}hue system with decidable word problem and without
  equivalent finite canonical system.
\newblock {\em Theoret. Comput. Sci.}, 35(2-3):337--344, 1985.

\bibitem[KN94]{KashiwaraNakashima94}
Masaki Kashiwara and Toshiki Nakashima.
\newblock Crystal graphs for representations of the {$q$}-analogue of classical
  {L}ie algebras.
\newblock {\em J. Algebra}, 165(2):295--345, 1994.

\bibitem[Knu70]{Knuth70}
Donald~E. Knuth.
\newblock Permutations, matrices, and generalized {Y}oung tableaux.
\newblock {\em Pacific J. Math.}, 34:709--727, 1970.

\bibitem[KO14]{KubatOkninski14}
{\L}ukasz Kubat and Jan Okni{\'n}ski.
\newblock Gr\"obner-{S}hirshov bases for plactic algebras.
\newblock {\em Algebra Colloq.}, 21(4):591--596, 2014.

\bibitem[Kob90]{Kobayashi90}
Yuji Kobayashi.
\newblock Complete rewriting systems and homology of monoid algebras.
\newblock {\em J. Pure Appl. Algebra}, 65(3):263--275, 1990.

\bibitem[Lec02]{Lecouvey02}
C{\'e}dric Lecouvey.
\newblock Schensted-type correspondence, plactic monoid, and jeu de taquin for
  type {$C_n$}.
\newblock {\em J. Algebra}, 247(2):295--331, 2002.

\bibitem[Lec03]{Lecouvey03}
C{\'e}dric Lecouvey.
\newblock Schensted-type correspondences and plactic monoids for types {$B_n$}
  and {$D_n$}.
\newblock {\em J. Algebraic Combin.}, 18(2):99--133, 2003.

\bibitem[Lit96]{Littelmann96}
Peter Littelmann.
\newblock A plactic algebra for semisimple {L}ie algebras.
\newblock {\em Adv. Math.}, 124(2):312--331, 1996.

\bibitem[LLT95]{LascouxLeclercThibon95}
Alain Lascoux, Bernard Leclerc, and Jean-Yves Thibon.
\newblock Crystal graphs and {$q$}-analogues of weight multiplicities for the
  root system {$A_n$}.
\newblock {\em Lett. Math. Phys.}, 35(4):359--374, 1995.

\bibitem[Lop14]{Lopatkin14}
V.~Lopatkin.
\newblock {Cohomology Rings of Plactic Monoid Algebra via {G}r\"obner-Shirshov
  Basis}.
\newblock {\em ArXiv e-prints}, November 2014.

\bibitem[Lot02]{Lothaire02}
M.~Lothaire.
\newblock {\em Algebraic combinatorics on words}, volume~90 of {\em
  Encyclopedia of Mathematics and its Applications}.
\newblock Cambridge University Press, Cambridge, 2002.

\bibitem[LS81]{LascouxSchutzenberger81}
Alain Lascoux and Marcel-P. Sch{\"u}tzenberger.
\newblock Le mono\"\i de plaxique.
\newblock In {\em Noncommutative structures in algebra and geometric
  combinatorics ({N}aples, 1978)}, volume 109 of {\em Quad. ``Ricerca Sci.''},
  pages 129--156. CNR, Rome, 1981.

\bibitem[Sch61]{Schensted61}
C.~Schensted.
\newblock Longest increasing and decreasing subsequences.
\newblock {\em Canad. J. Math.}, 13:179--191, 1961.

\bibitem[She99]{Sheats99}
Jeffrey~T. Sheats.
\newblock A symplectic jeu de taquin bijection between the tableaux of {K}ing
  and of {D}e {C}oncini.
\newblock {\em Trans. Amer. Math. Soc.}, 351(9):3569--3607, 1999.

\bibitem[SOK94]{Squier94}
Craig~C. Squier, Friedrich Otto, and Yuji Kobayashi.
\newblock A finiteness condition for rewriting systems.
\newblock {\em Theoret. Comput. Sci.}, 131(2):271--294, 1994.

\end{thebibliography}
\end{small}
\vfill

\noindent \textsc{Nohra Hage}\\
\begin{small}
Universit\'e de Lyon,\\
Institut Camille Jordan, CNRS UMR 5208\\
Universit\'e Jean Monnet\\
23, boulevard du docteur Paul Michelon,\\
42000 Saint-Etienne cedex, France.\\
\textsf{ nohra.hage@univ-st-etienne.fr}
\end{small}

\end{document}